\documentclass[10pt,reqno]{amsart}
\usepackage{eucal,fullpage,times,amsmath,amsthm,amssymb,mathrsfs,stmaryrd,color,enumerate,accents}
\usepackage[all]{xy}
\usepackage{url}
\usepackage[colorlinks=true,hyperindex,linkcolor=purple, citecolor=cyan, urlcolor=cyan, pagebackref=true]{hyperref}
\usepackage{tikz,caption}
\usetikzlibrary{decorations.markings}
\usepackage{pifont}

\usetikzlibrary{shapes.callouts}
\tikzset{
  notice/.style  = { draw, rectangle callout, callout relative pointer={#1} },
}

\tikzset{->-/.style={decoration={
  markings,
  mark=at position #1 with {\arrow{>}}},postaction={decorate}}}

\newcommand{\calO}{{\mathcal{O}}}
\newcommand{\calC}{\mathcal{C}}
\newcommand{\calU}{\mathcal{U}}
\newcommand{\calH}{\mathcal{H}}

\newcommand{\calX}{\mathcal{X}}
\newcommand{\frakX}{\mathfrak{X}}
\newcommand{\frakY}{\mathfrak{Y}}
\newcommand{\frakU}{\mathfrak{U}}

\newcommand{\Z}{\mathbf{Z}}

\newcommand{\N}{\mathbf{N}}
\newcommand{\C}{\mathbf{C}}
\newcommand{\F}{\mathbf{F}}
\newcommand{\Q}{\mathbf{Q}}
\newcommand{\A}{\mathbf{A}}
\newcommand{\crys}{\mathrm{crys}}
\renewcommand{\inf}{\mathrm{inf}}
\newcommand{\Spec}{{\mathrm{Spec}}}
\newcommand{\Spf}{{\mathrm{Spf}}}
\newcommand{\Alg}{\mathrm{Alg}}
\newcommand{\Shv}{\mathrm{Shv}}
\newcommand{\Hom}{\mathrm{Hom}}
\newcommand{\Ext}{\mathrm{Ext}}
\newcommand{\perf}{\mathrm{perf}}

\newcommand{\Mod}{\mathrm{Mod}}
\newcommand{\et}{\mathrm{\acute{e}t}}
\newcommand{\proet}{{\mathrm{pro}\et}}
\newcommand{\Zar}{\mathrm{Zar}}
\newcommand{\an}{\mathrm{an}}
\newcommand{\id}{\mathrm{id}}

\newcommand{\can}{\mathrm{can}}
\renewcommand{\ker}{\mathrm{ker}}
\newcommand{\Perf}{\mathrm{Perf}}

\newcommand{\fram}{\mathfrak{m}}

\newcommand{\cosimp}[3]{\xymatrix@1{#1 \ar@<.4ex>[r] \ar@<-.4ex>[r] & {\ }#2 \ar@<0.8ex>[r] \ar[r] \ar@<-.8ex>[r] & {\ } #3 \ar@<1.2ex>[r] \ar@<.4ex>[r] \ar@<-.4ex>[r] \ar@<-1.2ex>[r] & \cdots }}

\newcommand{\colim}{\mathop{\mathrm{colim}}}

\newcommand{\adjunction}[4]{\xymatrix@1{#1{\ } \ar@<0.3ex>[r]^{ {\scriptstyle #2}} & {\ } #3 \ar@<0.3ex>[l]^{ {\scriptstyle #4}}}}

\begin{document}

\newtheorem{theorem}{Theorem}[section]
\newtheorem*{theorem*}{Theorem}
\newtheorem*{definition*}{Definition}
\newtheorem{proposition}[theorem]{Proposition}
\newtheorem{lemma}[theorem]{Lemma}
\newtheorem{corollary}[theorem]{Corollary}

\theoremstyle{definition}
\newtheorem{definition}[theorem]{Definition}
\newtheorem{question}[theorem]{Question}
\newtheorem{remark}[theorem]{Remark}
\newtheorem{warning}[theorem]{Warning}
\newtheorem{example}[theorem]{Example}
\newtheorem{notation}[theorem]{Notation}
\newtheorem{convention}[theorem]{Convention}
\newtheorem{construction}[theorem]{Construction}
\newtheorem{claim}[theorem]{Claim}

\title{Specializing varieties and their cohomology \\ from characteristic $0$ to characteristic $p$}
\author{Bhargav Bhatt}
\begin{abstract}
We present a semicontinuity result, proven in recent joint work with Morrow and Scholze, relating the mod $p$ singular cohomology of a smooth projective complex algebraic variety $X$ to the de Rham cohomology of a smooth characteristic $p$ specialization of $X$: the rank of the former is bounded above by that of the latter. The path to this result passes through $p$-adic Hodge theory and perfectoid geometry, so we survey the relevant aspects of those subjects as well.
\end{abstract}
\maketitle

\section{Introduction}
Let $f:X \to S$ be a proper smooth morphism of schemes. If $S$ is a complex algebraic variety, then the singular cohomology $H^i(X_s^{\an},\Z)$ of the fiber $X_s^{\an}$ over a closed point $s \in S(\C)$ is independent of $s$. A similar assertion holds if $S$ is an algebraic variety in positive characteristic $p$ once one replaces the singular cohomology group $H^i(X_s^{\an},\Z)$ with the $\ell$-adic \'etale cohomology group $H^i_{\et}(X_s,\Z_\ell)$ for some prime $\ell \neq p$. In contrast, for $p$-adic cohomology theories such as crystalline cohomology, the analogous result is often false: the ``pathologies'' of characteristic $p$ geometry furnish examples\footnote{We are not aware of an example of this phenomenon in the literature. However, one can obtain an example by a slight modification of \cite[\S 2.2]{BMSMainPaper} as follows: let $S = \Spec(\overline{\F_p}\llbracket t \rrbracket)$ with generic point $\eta$ and closed point $s$, let $G/S$ be a finite flat commutative group scheme whose generic fibre is $\Z/p^2$ and whose special fibre is $\alpha_p \times \alpha_p$, and let $X \to S$ be a smooth projective approximation of the Artin stack $BG \to S$; then one can show that $H^1_{dR}(X_\eta)$ has dimension $1$, while $H^1_{dR}(X_s)$ has dimension $2$.} where the $p$-torsion in $H^i_\crys(X_s)$ (or, essentially equivalently, the rank of the algebraic de Rham cohomology group $H^i_{dR}(X_s)$) can vary with $s$. The best one can say in general is that the ranks of $H^i_\crys(X_s)[\frac{1}{p}]$ are independent of $s$, and the ranks of $H^i_{dR}(X_s)$ are upper semicontinuous in $s$. The recent paper \cite{BMSMainPaper} paints a similar picture when $S$ is the spectrum of a mixed characteristic $(0,p)$ valuation ring: the torsion in the crystalline cohomology $H^i_\crys(X_s)$ of the characteristic $p$ fibre gives an upper bound for the torsion in the \'etale cohomology $H^i_{\et}(X_{\bar{\eta}},\Z_p)$ of the (geometric) characteristic $0$ fibre (and this upper bound can be strict). This extends previous results by many people, notably Faltings \cite{FaltingsRamified} and Caruso \cite{Caruso}. 

The goal of these notes is to give a slightly simpler account of the mod $p$ version of the aforementioned result from \cite{BMSMainPaper}, i.e., in the mixed characteristic setting, we relate the de Rham cohomology $H^i_{dR}(X_s)$ of the special fibre to the $\F_p$-\'etale cohomology $H^i_{\et}(X_{\bar{\eta}},\F_p)$ of the generic fibre. In contrast to the equicharacteristic cases, the cohomology theories attached to the two fibers  in the mixed characteristic case have very different origins; in relating them, we will naturally encounter (and review) some key facets of $p$-adic Hodge theory.

\subsection{Statement of the main theorem}

We now formulate the semicontinuity result more precisely. Fix a prime number $p$, and let $C$ be a complete and algebraically closed nonarchimedean extension of $\Q_p$; write $\calO_C \subset C$ for the valuation ring, and let $k$ be the residue field. The main goal of these notes is to sketch a proof of the following theorem:

\begin{theorem}
\label{thm:MainThm}
Let $\frakX/\calO_C$ be a proper smooth scheme. Then we have the inequality
\[ \dim_{\F_p} H^i_\et(\frakX_C, \F_p) \leq \dim_k H^i_{dR}(\frakX_k/k).\]
More generally, the same holds for proper smooth formal schemes $\frakX/\calO_C$ provided one interprets the generic fibre $\frakX_C$ and its \'etale cohomology in the sense of adic spaces, as in the work of Huber \cite{HuberContVal,HuberDefAdic,Huber}.
\end{theorem}

In particular, the presence of $p$-torsion in \'etale cohomology of the characteristic $0$ fibre forces the de Rham cohomology of the characteristic $p$ fibre to be larger than expected; more evocatively, the obstruction to ``integrating'' over a $p$-torsion class in singular homology is a differential form in characteristic $p$. This provides an explanation for some of the pathologies experimentally observed in de Rham cohomology in characteristic $p$; for example, Enriques surfaces in characteristic $2$ must have non-trivial $H^1_{dR}$ (first observed by direct calculation in \cite[Corollary 7.3.4]{IllusieDRWitt} using the classification in \cite{BombieriMumford3}) since their characteristic $0$ lifts have $\Z/2$ as their fundamental groups. We refer to \cite[\S 1]{BMSMainPaper} for more on the history of previous work relating \'etale and de Rham cohomology in the $p$-adic setting.

\subsection{Strategy of the proof and its relation to \cite{BMSMainPaper}}
\label{ss:IntroRelateBMS}

The inequality above is a consequence of the more refined \cite[Theorem 1.8]{BMSMainPaper}, and the the proof given here follows the same skeleton as the one in {\em loc.\ cit.}: one attaches to $\frakX$ a cohomology theory $R\Gamma_{\calO_C^\flat}(\frakX)$ that is a perfect complex over a characteristic $p$ valuation ring $\calO_C^\flat$ with generic fibre given (essentially) by the $\F_p$-\'etale cohomology of $\frakX_C$ and special fibre given by the de Rham cohomology of $\frakX_k$; this yields the desired inequality by semicontinuity. In fact, in both cases, the cohomology theory $R\Gamma_{\calO_C^\flat}(\frakX)$ is the hypercohomology of (the mod $p$ reduction of) a certain complex $A\Omega_\frakX$ on $\frakX$, and the bulk of the work lies in analysing $A\Omega_\frakX$ locally on $\frakX$ using techniques of ``almost mathematics'' in the sense of Faltings \cite{FaltingsJAMS,FaltingsAlmostEtale,GabberRamero} and related ideas coming from perfectoid spaces \cite{ScholzeThesis,ScholzePAdicHodge}. The main difference is that, in these notes, since we are only after the global inequality above, we implement this strategy in a somewhat simplified fashion. More precisely:
\begin{enumerate}[(a)]

\item The paper \cite{BMSMainPaper} gives a precise local description of $A\Omega_\frakX$. In these notes,  however, we content ourselves with working entirely within the realm of almost mathematics (but see part (e) below); this allows us to bypass some delicate arguments in \cite{BMSMainPaper} involved in showing certain almost zero modules are actually zero. 

\item The paper \cite{BMSMainPaper} identifies a specialization of $A\Omega_\frakX$ and the de Rham-Witt complex of $\frakX_k$, thus bringing the crystalline cohomology of $\frakX_k$ into the fray. On the other hand, in these notes, we completely ignore the connection to de Rham-Witt complexes and crystalline cohomology; this is necessitated by our desire to avoid any technical baggage not relevant to Theorem~\ref{thm:MainThm}, 	and we refer the interested reader to \cite{MorrowNotesIpHT} for more on this connection.

\item The paper \cite{BMSMainPaper} contains a detailed discussion of the category in which the cohomology of the complex $A\Omega_\frakX$ takes values, i.e., Breuil-Kisin-Fargues modules. A good understanding of this category is necessary for applications relating crystalline and \'etale cohomology with their concomitant structure, especially those involving the recovery of one from the other (such as \cite[Theorem 1.4]{BMSMainPaper}). However, such an understanding is not necessary for geometric applications such as Theorem~\ref{thm:MainThm}, so these notes avoid this discussion completely. 

\item Theorem~\ref{thm:MainThm} does not involve any almost mathematics; indeed, working entirely in the almost setting would kill all $k$-vector spaces, so it wouldn't be possible to say anything about $H^i_{dR}(\frakX_k/k)$. Nevertheless, we succeed in deducing Theorem~\ref{thm:MainThm} from an almost description of $A\Omega_\frakX$ by using a trick involving spherical completions. In particular, in these notes, we only construct the cohomology theory $R\Gamma_A(\frakX)$ when $C$ is spherically complete (which suffices for applications such as Theorem~\ref{thm:MainThm}).

\item In these notes, we give an alternate, and somewhat faster and cleaner, approach to the main results of \cite[\S 9]{BMSMainPaper} that involve identifying $A\Omega_\frakX$ in the literal (and not merely almost) sense. This relies ultimately on an observation about the $L\eta$-functor (see Lemma~\ref{lem:LetaRegularSequence}) that was missed in \cite{BMSMainPaper}, and is spelled out in Remarks~\ref{rmk:TildeOmegaNoAlmost}, \ref{rmk:AOmegaNoAlmost}, and \ref{rmk:RGammaANoAlmost}; these remarks are not relevant for the global applications such as Theorem~\ref{thm:MainThm}.

\item In \cite{BMSMainPaper}, the \'etale comparison theorem for $A\Omega_\frakX$ was deduced from the primitive comparison theorem. In these notes, we recall this proof, but also provide a direct proof using the de Rham comparison theorem (which we prove) and standard facts in $p$-linear algebra (see Remark~\ref{rmk:NoPrimitive}).
\end{enumerate}

\subsection{Outline}
We begin in \S \ref{sec:PerfectPerfectoid} by recalling some basic notions on perfect and perfectoid rings; the main aim is to explain the geometry of Fontaine's ring $A_\inf$ (the most fundamental $p$-adic period ring), and use it to formulate a more precise theorem implying Theorem~\ref{thm:MainThm}. In \S \ref{sec:Almost}, we introduce some basic notions from almost mathematics, including a slightly nonstandard version of almost mathematics over Fontaine's ring $A_\inf$, and prove a lemma about how this behaves in the spherically complete case. The relevant background from the perfectoid approach to $p$-adic Hodge theory is summarized in \S \ref{sec:FramedAlgebras}; this includes, in particular, some notation surrounding perfectoid tori over $\calO_C$ and their \'etale covers that will be used repeatedly later. One of the main innovations of \cite{BMSMainPaper} --- the Berthelot-Ogus functor $L\eta_f(-)$ and its ability to kill torsion in derived categories --- is then introduced and studied in \S \ref{sec:Leta}. These tools enable us to define and study the complex $A\Omega_\frakX$ in \S \ref{sec:AOmega}; to make this study flow smoothly, we first discuss a certain specialization $\widetilde{\Omega}_\frakX$ of $A\Omega_\frakX$ in \S \ref{sec:TildeOmega}. Finally, the relevant global consequences for Theorem~\ref{thm:MainThm} are deduced in \S \ref{sec:globalAOmega}.

\subsection{Conventions}
\label{ss:Conventions} We explain the (largely standard) conventions followed in these notes.

\subsubsection*{Derived categories}
All rings will be commutative with $1$. For a ring $R$, write $\Mod_R$ for the category of $R$-modules, and $D(R)$ for the (unbounded) derived category of $R$-modules. We shall use implicitly use the fact that standard functors (such as tensor products, $\mathrm{Hom}$, limits, etc.) admit well-behaved derived functors at the level of unbounded derived categories; this was worked out by Spaltenstein \cite{Spaltenstein}, and a convenient modern reference is \cite[Tag 05QI]{StacksProject}. We shall identify an $R$-module $M$ with the chain complex $M[0]$ obtained by placing $M$ in degree $0$. If $f \in R$ is a nonzerodivisor and $K \in D(R)$, we will often write $K/f$ for $K \otimes^L_R R/f$ if there is no confusion. Also, we use the following fact without comment: if $K[\frac{1}{f}] = 0$ and $K/f = 0$, then $K = 0$. Indeed, the second condition implies $f$ acts invertibly on $K$, whence the first condition implies $K = 0$.

\subsubsection*{Completions}
Given a ring $R$ and a finitely generated ideal $I = (f_1,...,f_r) \subset R$, we will often talk about objects $M \in D(R)$ that are $I$-adically complete; this is always meant in the derived sense. Recall that $M$ is {\em $I$-adically complete} iff the natural maps induce an isomorphism $M \simeq R\lim_n(M \otimes^L_{\Z[x_1,...,x_r]} \Z[x_1,...,x_r]/(x_i^n))$, where $x_i \in \Z[x_1,...,x_r]$ acts via $f_i$ on $M$; this is equivalent to asking that the (derived) inverse limit of the tower $\{ ... \to M \stackrel{f_i}{\to} M \stackrel{f_i}{\to} M\}$ vanishes for each $f_i$. Such complexes form a full triangulated subcategory of $D(R)$ and have the following features: 
\begin{enumerate}
\item A complex $M$ is complete if and only if each $H^i(M)$ is complete.
\item The complete complexes which are discrete (i.e., have cohomology only in degree $0$) form an abelian subcategory of all $R$-modules.
\item Nakayama's lemma: a complete complex $M$ is $0$ if and only if $M \otimes^L_{\Z[x_1,...,x_r]} \Z$ is so. 
\item The inclusion of complete complexes into $D(R)$ has a left-adjoint $M \mapsto \widehat{M}$ called the {\em completion} functor, which is explicitly computed by the formula 
\[ \widehat{M} := R\lim_n(M \otimes^L_{\Z[x_1,...,x_r]} \Z[x_1,...,x_r]/(x_i^n)).\]
Thus, any $M \in D(R)$ admits a canonical map $M \to \widehat{M}$ which is an isomorphism exactly when $M$ is complete. 
\item Fix an $R$-module $M$ that is $I$-adically separated. Then $M$ is $I$-adically complete in the classical sense (i.e., $M \simeq \lim M/I^n M$) if and only if $M$ is complete when regarded as a complex. 
\item Say $I = (f)$ and $M$ is an $R$-module on which $f$ acts injectively. Then the completion $\widehat{M}$ of $M$ as a complex coincides with the classical $I$-adic completion $\lim_n M/I^n M$. In particular, such an $M$ is $I$-adically complete in the classical sense if and only if $M$ is complete as a complex.
\end{enumerate}
We refer to \cite[\S 6.2]{BMSMainPaper}, \cite[\S 3.4]{BhattScholze} and \cite[Tag 091N]{StacksProject} for a more complete discussion. 

\subsubsection*{Koszul complexes}
For an abelian group $M$ equipped with commuting endomorphisms $f_1,...,f_r$, we often consider the Koszul complex
\[ K(M;f_1,...,f_r) := M \to M \otimes_{\Z} \Z^{\oplus r} \to M \otimes_{\Z} \wedge^2 (\Z^{\oplus r}) \to .... \to M \otimes_{\Z}\wedge^r (\Z^{\oplus r}),\]
viewed as a chain complex in cohomological degrees $0,...,r$; this chain complex calculates the object $R\Hom_{\Z[x_1,...,x_r]}(\Z, M)$ where $x_i$ acts by $f_i$ on $M$ and trivially on $\Z$. Now assume that $M$ is an $R$-module for some ring $R$, and the $f_i$'s are elements of $R$. We will often use the following observations: (a) the complex $K(M;f_1,...,f_r)$ admits the structure of a complex over the simplicial commutative ring\footnote{The theory of simplicial commutative rings up to homotopy is developed by Quillen \cite{QuillenHA}, and geometrized to ``derived algebraic geometry'' in the work of Lurie \cite{LurieThesis} (see also \cite{LurieSAG} for more). We do not use any non-formal input from this theory.} $\Z \otimes^L_{\Z[x_1,...,x_r]} R$, (b) each homology group of this complex is an $R$-module annihilated by each $f_i$, and (c) if $g \in (f_1,...,f_r)$ is a nonzerodivisor, then $K(M;f_1,...,f_r)$ admits the structure of a complex over $R/(g) \simeq \Z \otimes^L_{\Z[y]} R$, where $y$ maps to $0$ in $\Z$ and $g \in R$.

\subsubsection*{Others}
The letter $C$ will be reserved to denote a complete and algebraically closed nonarchimedean field; here the valuation is always required to have rank $1$. When working over a ring $A$ equipped with a notion of `almost' isomorphisms, given $A$-modules $M$ and $N$, we write $M \stackrel{a}{\simeq} N$ for an almost isomorphism between $M$ and $N$. 

\subsection*{Acknowledgements} I would like to thank the organizers of the 2015 AMS Algebraic Geometry Symposium for encouraging me to prepare these notes on the joint work \cite{BMSMainPaper}. I'd also like to thank my collaborators Matthew and Peter for many enlightening discussions about this project; the novel ideas presented here, including those summarized in \S \ref{ss:IntroRelateBMS}, were conceived in these exchanges. I'm grateful to Jean-Marc Fontaine for a valuable discussion about $A_\inf$, and to Jacob Lurie for numerous conversations about $A\Omega_\frakX$. Thanks are due to Wei Ho, Jacob Lurie, and Peter Scholze for their many comments that helped significantly improve this writeup. I am also indebted to Kestutis \v{C}esnavi\v{c}us as well as the two anonymous referees for their thorough reading of a previous version of these notes and numerous comments. During the preparation of these notes, I was supported by NSF Grant DMS \#1501461 and a Packard fellowship.

\section{Perfect and perfectoid rings}
\label{sec:PerfectPerfectoid}

The main goal of this section is to introduce Fontaine's period ring $A_\inf$ in \S \ref{ss:PerfectoidField}, explain what it looks like in Figure~\ref{fig:Ainf}, and use it to give a better formulation of the main theorem proven in these notes in \S \ref{ss:BetterMainThm}. The input to the construction of $A_\inf$ is the perfectoid nature of $\calO_C$, so we spend some time in \S \ref{ss:PerfectoidRings} developing some language to study perfectoid rings and Fontaine's $A_\inf$-functor in general. In particular, we give an introduction to Scholze's tilting correspondence \cite{ScholzeThesis} in Remark~\ref{rmk:Tilting}. The main input in the presentation of the tilting correspondence here is the vanishing of the cotangent complex of perfectoid rings. This vanishing is a mixed characteristic analog of the vanishing of the cotangent complex of perfect rings in characteristic $p$, so we spend some time in \S \ref{ss:PerfectRings} studying perfect rings and discussing consequences of the cotangent complex vanishing, such as a well-behaved Witt vector functor.

\subsection{Perfect rings}
\label{ss:PerfectRings}

Fix a prime $p$. We will study the following class of $\F_p$-algebras:

\begin{definition}
An $\F_p$-algebra $R$ is {\em perfect} if the Frobenius map $R \to R$ is an isomorphism. Let $\Perf_{\F_p}$ denote the category of perfect $\F_p$-algebras.
\end{definition}

The following two constructions of perfect rings out of ordinary rings are useful in the sequel:

\begin{example}[Perfections]
Given any $\F_p$-algebra $R$, one obtains two canonically associated perfect rings by the formulas
\[ R_\perf := \colim_{x \mapsto x^p} R \quad \mathrm{and} \quad R^\perf := \lim_{x \mapsto x^p} R. \]
The canonical map $R \to R_\perf$ (resp. $R^\perf \to R$) is the universal map from $R$ to a perfect ring (resp. to $R$ from a perfect ring). Some explicit examples that are relevant for the sequel are given as follows:
\begin{enumerate}
\item Let $R = \F_p[x]$. Then $R_\perf = \F_p[x^{\frac{1}{p^\infty}}]$ is the set of all polynomials over $\F_p$ with exponents in $\N[\frac{1}{p}]$, and $R^\perf \simeq \F_p$ is the set of constant polynomials in $R$.
\item Let $S = \F_p[x^{\frac{1}{p^\infty}}]/(x)$. Then $S_\perf = \F_p$, and $S^\perf$ is the $x$-adic completion of $\F_p[x^{\frac{1}{p^\infty}}]$.
\item Let $R$ be a finite type $k$-algebra with $k$ a perfect field of characteristic $p$. Then $R^\perf \simeq k^{\pi_0(\Spec(R))}$.
\end{enumerate}
In particular, examples (2) and (3) illustrate an important feature: the functor $(-)^\perf$ tends to be quite lossy unless one restricts attention to {\em semiperfect} rings, i.e., rings where Frobenius is surjective; such rings will have many nilpotents (unless they are themselves perfect). In contrast, the functor $(-)_\perf$ completely destroys all nilpotents.
\end{example}

An essential feature of perfect rings is that they admit a canonical (and unique) ``one parameter deformation'' across $\Z_p \to \F_p$. Such a deformation can be provided explicitly by the Witt vector functor (see \cite[\S II.6]{SerreLocalFields} for the standard construction), but we take a slightly more abstract perspective to summarize the properties of this construction that are relevant for later applications: 

\begin{proposition}[Witt vectors of perfect rings]
\label{prop:WittVectorDescription}
Let $\widehat{\Alg}_{\Z_p}$ denote the category of $p$-adically complete and separated $p$-torsionfree $\Z_p$-algebras. The functor $\widehat{\Alg}_{\Z_p} \to \Perf_{\F_p}$ determined by $B \mapsto B^\flat := (B/p)^\perf$ admits a left adjoint $A \mapsto W(A)$.
\end{proposition}

The careful reader shall observe that our conventions imply that a $\mathbf{Z}_p$-module $M$ that is $p$-torsionfree and $p$-adically complete in the derived sense is automatically $p$-adically separated, so the extra hypothesis above is redundant. Nevertheless, we explicitly mention it to avoid any possible confusion.

\begin{proof}
We first construct the functor $W(-)$. For this, observe that the cotangent complex $L_{A/\F_p}$ vanishes: the Frobenius on $A$ induces a map $L_{A/\F_p} \to L_{A/\F_p}$ that is simultaneously $0$ (as this is true for any $\F_p$-algebra) and an isomorphism (as $A$ is perfect). In particular, for any infinitesimal extension $R \to \F_p$, there is a unique (up to unique isomorphism) flat map $R \to A_R$ lifting $\F_p \to A$. Applying this to $R = \Z/p^n$ defines $W_n(A) = A_{\Z/p^n}$. Taking limits gives $W(A) := \lim W_n(A)$, and one readily checks that $W(A) \in \widehat{\Alg}_{\Z_p}$, thus defining $W(-):\Perf_{\F_p} \to \widehat{\Alg}_{\Z_p}$. The adjunction between $W(-)$ and $(-)^\flat$ is an exercise in using the vanishing of the cotangent complex and the defining property of $(-)^\perf$.
\end{proof}

By functoriality, the Frobenius automorphism $\phi$ of a perfect ring $A$ induces an automorphism of $W(A)$ that is abusively also denoted by $\phi$. Note that this construction provides an abundant supply of characteristic $0$ rings (namely, $W(A)[\frac{1}{p}]$) equipped with a map that deserves to be called a ``Frobenius''.

\begin{remark}[Teichmuller maps]
\label{rmk:Teichmuller}
The construction $A \mapsto W(A)$ on $\Perf_{\F_p}$ enjoys the following lifting feature: the identity map $A \to A$ lifts uniquely to a multiplicative map $[-]:A \to W(A)$: send $a \in A$ to $\lim b_n^{p^n}$, where $b_n \in W(A)$ is a lift of $a^{\frac{1}{p^n}}$. Less explicitly, the fiber over $1$ of $W_{n+1}(A) \to W_n(A)$ is $p$-torsion; as multiplication by $p$ is a bijection on the multiplicative monoid underlying $A$, there are no obstructions/choices in lifting $\id:A \to A$ inductively up along each $W_{n+1}(A) \to W_n(A)$ to get a multiplicative map $[-]:A \to W(A)$ in the limit. The existence of this map immediately shows that any $f \in W(A)$ can be uniquely written as a power series $\sum_{i \geq 0} [a_i] \cdot p^i$ with $a_i \in W(A)$: given $f \in W(A)$, we set $a_0 \in A$ to be the image of $f$, and then inductively define $a_n$ by checking that $p^n \mid (f -\sum_{i=0}^{n-1} [a_i] p^i)$ and setting $a_n$ as the image of $\frac{f - \sum_{i=0}^{n-1} [a_i] p^i}{p^n}$ under the reduction map $W(A) \to A$. Thanks to this description, one can work with the Witt vectors very explicitly. In these notes, we have chosen to de-emphasize the explicit presentation in favor of the conceptual description via deformation theory.
\end{remark}

\begin{remark}[The Witt vectors as an equivalence of categories]
\label{rmk:WittVectorPerfectCategory}
Consider the full subcategory $\widehat{\Alg}^\perf_{\Z_p} \subset \widehat{\Alg}_{\Z_p}$ spanned by those rings $B$ for which $B/p$ is perfect. The proof of Proposition \ref{prop:WittVectorDescription} shows that $W(-)$ gives an equivalence $\Perf_{\F_p} \simeq \widehat{\Alg}^\perf_{\Z_p}$. Moreover, for any such $B \in \widehat{\Alg}^{\perf}_{\Z_p}$, we have $L_{B/\Z_p} \otimes^L_{\Z_p} \F_p \simeq L_{(B/p)/\F_p} \simeq 0$, and thus the $p$-adic completion $\widehat{L_{B/\Z_p}}$ (which, we recall, is computed as $R\lim_n (L_{B/\mathbf{Z}_p} \otimes^L_{\mathbf{Z}_p} \mathbf{Z}/p^n)$) vanishes.
\end{remark}

\begin{example}[The perfect polynomial ring]
\label{ex:PerfectWitt}
Let $R = \F_p[x]_\perf = \F_p[x^{\frac{1}{p^\infty}}]$. Then we claim that $W(R) = \widehat{\Z_p[x^{\frac{1}{p^\infty}}]}$, where the completion is $p$-adic. To see this, by Remark~\ref{rmk:WittVectorPerfectCategory}, it suffices to observe that $\widehat{\Z_p[x^{\frac{1}{p^\infty}}]}$ is an object of $\widehat{\Alg}_{\Z_p}^{\perf}$ that reduces mod $p$ to $R$. More generally, if $R = R_{0,\perf}$ for an $\F_p$-algebra $R_0$, and $\widetilde{R}_0$ is a $\Z_p$-flat lift of $R_0$ equipped with a lift $\phi:\widetilde{R}_0 \to \widetilde{R}_0$ of Frobenius, then the same argument shows that $W(R)$ is the $p$-adic completion of $\colim_\phi \widetilde{R}_0$.
\end{example}

\subsection{Perfectoid rings}
\label{ss:PerfectoidRings}

Fix a complete and algebraically closed nonarchimedean extension $C/\Q_p$. To a first approximation, a perfectoid ring can be viewed as an analog of a perfect ring over $\calO_C$. More precisely, one defines:

\begin{definition}[Scholze]
An $\calO_C$-algebra $R$ is said to be {\em perfectoid} if $R$ is $p$-adically complete, $p$-torsionfree, and Frobenius induces an isomorphism $R/p^{\frac{1}{p}} \simeq R/p$. Let $\Perf_{\calO_C}$ be the category of all such $\calO_C$-algebras. 
\end{definition}

\begin{remark}[Perfectoidness as relative perfectness]
\label{rmk:PerfectoidCC}
For a $p$-adically complete $p$-torsionfree $\calO_C$-algebra $R$, being perfectoid is equivalent to requiring that the relative Frobenius for $\calO_C/p \to R/p$ is bijective. In particular, a slight variant of the trick used in Proposition~\ref{prop:WittVectorDescription} shows that $L_{R/\calO_C} \otimes^L_R R/p \simeq L_{(R/p)/(\calO_C/p)}$ vanishes, and thus the $p$-adic completion $\widehat{L_{R/\calO_C}}$ also vanishes by Nakyama's lemma for complete complexes.
\end{remark}

\begin{remark}
In the literature, one finds a plethora of different notions of perfectoid rings, adapted to the problem at hand. In particular, the notion introduced above is sometimes called {\em integral perfectoid} to emphasize that $p$ is not invertible on the ring in question.
\end{remark}

In particular, the ring $\calO_C$ itself is perfectoid. The most imporant examples for us are $p$-adic analogs of solenoids:

\begin{example}[The perfectoid torus]
\label{ex:PerfectoidBasic}
The $\calO_C$-algebra $\calO_C \langle t_1^{\pm \frac{1}{p^\infty}},..., t_d^{\pm \frac{1}{p^\infty}} \rangle$, i.e., the $p$-adic completion of $\calO_C[t_1^{\pm \frac{1}{p^\infty}},....,t_d^{\pm \frac{1}{p^\infty}}]$, is perfectoid. More generally, the $p$-adic completion of any \'etale $\calO_C \langle t_1^{\pm \frac{1}{p^\infty}},..., t_d^{\pm \frac{1}{p^\infty}} \rangle$-algebra is perfectoid; here one uses that Frobenius base changes to Frobenius along an \'etale map.
\end{example}

In analogy with the Witt vector functor on perfect rings, there is an analogous ``one parameter deformation'' of a perfectoid ring given by Fontaine's $A_\inf(-)$ functor:

\begin{definition}[Fontaine]
\label{def:Ainf}
For a perfectoid $\calO_C$-algebra $R$, define 
\[ R^\flat := \lim_{x \mapsto x^p} R/p \quad \mathrm{and} \quad A_\inf(R) := W(R^\flat).\]
The ring $R^\flat$ is called the {\em tilt} of $R$ (following Scholze). We write $\phi$ for the Frobenius automorphism of $R^\flat$ or $A_\inf(R)$. For $R = \calO_C$ itself, we simply write $A_\inf := A_\inf(\calO_C)$.
\end{definition}

\begin{remark}[Explaining the name]
The notation $A_\inf(-)$ is meant to be suggestive: for a perfectoid $\calO_C$-algebra $R$, the ring $A_\inf(R)$ is the universal pro-infinitesimal thickening of $\Spec(R)$ relative to $\Z_p$ (see \cite[\S 1.2]{Fontainepadicperiods}), i.e., for any infinitesimal thickening $R' \to R$, there is a unique map $A_\inf(R) \to R'$ compatible with the map $\theta$ from Lemma~\ref{lem:AinfTheta} below. In particular, there is a canonical isomorphism 
\[ R\Gamma( (\Spec(R)/\Z_p)_\inf, \calO_\inf) \simeq A_\inf(R).\]
Since we do not need this later, we do not prove this assertion here.
\end{remark}

The set $R^\flat$ can be described in a more ``strict'' fashion as follows:

\begin{lemma}
\label{lem:FlatPerfectoidAlternate}
Let $R$ be a perfectoid $\calO_C$-algebra. Define the multiplicative monoid 
\[ R^{\flat,'} := \lim_{x \mapsto x^p} R.\]
The natural map
\[ \alpha:R^{\flat,'} := \lim_{x \mapsto x^p} R \to R^\flat := \lim_{x \mapsto x^p} R/p\]
given by reduction modulo $p$ on terms is a bijection of multiplicative monoids. 
\end{lemma}

As the proof below shows, the conclusion of this lemma is valid for any $p$-adically complete ring $R$.

\begin{proof}
Write elements in $R^\flat$ as sequences $(a_0,a_1,a_2,...)$ with $a_i \in R/p$ and $a_{i+1}^p = a_i$, and similarly for $R^{\flat,'}$. For injectivity of $\alpha$: if $(a_i),(b_i) \in R^{\flat,'}$ with $a_i = b_i \mod p$ for all $i$, then one inductively shows that $a_{i+n}^{p^n} = b_{i+n}^{p^n} \mod p^{n+1}$ for all $i,n$, and thus $a_i = b_i \mod p^n$ for all $i,n$, which proves $a_i = b_i$ for all $i$ by $p$-adic completeness. For surjectivity of $\alpha$: given $(a_i) \in R^\flat$ and arbitrary lifts $\tilde{a}_i \in R$ of $a_i$, one checks that the sequence $n \mapsto \tilde{a}_{i+n}^{p^n}$ converges for all $i$, and thus setting $b_i := \lim_n \tilde{a}_{i+n}^{p^n} \in R$ gives an element $(b_i) \in R^{\flat,'}$ that lifts $(a_i)$. 
\end{proof}

We now justify why $A_\inf(R)$ may be considered as a ``one parameter deformation'' of $R$:

\begin{lemma}[Fontaine's map $\theta$]
\label{lem:AinfTheta}
Fix some $R \in \Perf_{\calO_C}$. The canonical projection $\overline{\theta}:R^\flat \to R/p$ fits into a unique pushout square of commutative rings
\begin{equation}
\label{eq:Ainf}
\xymatrix{ A_\inf(R) \ar[r]^-{\theta} \ar[d]^-{\mathrm{kill\ p}} & R \ar[d]^-{\mathrm{kill\ p}} \\
			R^\flat \ar[r]^-{\overline{\theta}} & R/p.}
\end{equation}
Both $\theta$ and $\overline{\theta}$ are surjective, and the kernel of either is generated by a nonzerodivisor.
\end{lemma}
\begin{proof}
The existence and uniqueness of the  commutative square with the above maps is immediate from the adjunction in Proposition~\ref{prop:WittVectorDescription}. The square is a pushout square as both vertical maps can be identified with reduction modulo $p$. Now $\overline{\theta}$ is surjective as $R$ is integral perfectoid; the $p$-adic completeness of the top row then implies that $\theta$ is also surjective.

Next, we identify $\ker(\overline{\theta})$ using the multiplicative bijection $\alpha:R^{\flat,'} \simeq R$ from Lemma~\ref{lem:FlatPerfectoidAlternate}. Choose a compatible system of $p$-power roots of $p$ in $\calO_C$, viewed as an element $\underline{p} := (p,p^{\frac{1}{p}},...) \in \calO_C^{\flat,'} \stackrel{\alpha}{\simeq} \calO_C^\flat$. We first observe that (a) $\underline{p}$ is a nonzerodivisor, and (b) it generates the kernel of $\overline{\theta}$. For $(a)$, using the multiplicative structure of $\alpha$, it is enough to observe that $p$ is a nonzerodivisor on $R$ by $\calO_C$-flatness. For (b), if $(a_i) \in R^{\flat,'}$ gives an element of $\ker(\overline{\theta})$ under $\alpha$, then $a_0 \in (p)$; computing valuations and using that $a_{i+1}^p = a_i$, we learn that $a_i \in (p^{\frac{1}{p^i}})$ for all $i$. As $R$ is $\calO_C$-flat, this shows that $(a_i) \in R^\flat$ is divisible by $\underline{p}$, as wanted.

Finally, we identify $\ker(\theta)$. For this, consider the commutative diagram
\[ \xymatrix{ 0 \ar[r] & \ker(\theta) \ar[r] \ar[d]^-{\beta} & A_\inf(R) \ar[r]^-{\theta} \ar[d]^-{\mathrm{kill\ p}} & R \ar[d]^-{\mathrm{kill\ p}} \ar[r] & 0 \\
		   0 \ar[r] & \ker(\overline{\theta}) \ar[r] &	R^\flat \ar[r]^-{\overline{\theta}} & R/p \ar[r] & 0}\]
Here the second row is obtained by applying $(-) \otimes_{A_\inf(R)} R^\flat$ to the first row, and using that $R$ is $p$-torsionfree. Now choose any $\xi \in \ker(\theta)$ that maps to $\underline{p} \in \ker(\overline{\theta})$; for example, we may take $\xi = p - [\underline{p}]$ (by Remark~\ref{rmk:TeichmullerPerfectoid} below, we have $\theta([\underline{p}]) = p$). Then we claim that (a) $\xi$ is a nonzerodivisor, and (b) $\xi$ generates $\ker(\theta)$. For (a), choose some $0 \neq f \in A_\inf(R)$ such that $f \cdot \xi = 0$ in $A_\inf(R)$. Dividing $f$ by a suitable power of $p$, since $A_\inf$ is $p$-adically separated, we may assume that $f \notin (p)$; as $p$ is a nonzerodivisor, we still have $f \cdot \xi = 0$. But then $\xi$ is a zerodivisor modulo $p$, which is a contradiction. For (b), one simply observes that $(\xi) \subset \ker(\theta)$ is an inclusion of $p$-adically complete flat $\Z_p$-modules which is an isomorphism modulo $p$ by the previous paragraph, and thus an isomorphism by completeness.
\end{proof}

\begin{remark}[The map $\tilde{\theta}$]
Fix $R \in \Perf_{\calO_C}$. Then Lemma~\ref{lem:AinfTheta} gives us a map $\theta:A_\inf(R) \to R$. Using the Frobenius automorphism $\phi$ of $A_\inf(R)$, we obtain many more maps $A_\inf(R) \to R$. In particular, the map
\[ \tilde{\theta} := \theta \circ \phi^{-1}:A_\inf(R) \to R\]
will play a crucial role in the sequel.
\end{remark}

Some examples of the $A_\inf(-)$ construction are recorded next:

\begin{example}[The tilt of a point]
\label{ex:PerfectoidField}
Let $R = \calO_C$. Choose the element $\varpi := \underline{p} \in \calO_C^\flat$ as in the proof of Lemma~\ref{lem:AinfTheta}, so we have $\calO_C^\flat/\varpi \simeq \calO_C/p$ by construction. The ring $\calO_C^\flat$ turns out, by a fundamental result of Fontaine-Wintenberger, to be the valuation ring of a complete and algebraically closed nonarchimedean field $C^\flat := \calO_C^\flat[\frac{1}{\varpi}]$ of characteristic $p$; any nonzero noninvertible element of $\calO_C^\flat$, such as the element $\varpi$, is called a {\em pseudouniformizer}. In the special case $C = \C_p$, one can be more explicit: the map $\F_p[\varpi] \to \calO_C^\flat$ identifies $\calO_C^\flat$ with the $\varpi$-adically completed absolute integral closure $\widehat{\overline{\F_p \llbracket \varpi \rrbracket}}$ of the source. The structure of $\calO_C^\flat$ and $A_\inf(\calO_C)$ is studied further in \S \ref{ss:PerfectoidField}.
\end{example}

\begin{example}[The tilt of a torus]
\label{ex:PerfectoidTorusAinf}
Let $R = \calO_C \langle t_1^{\pm \frac{1}{p^\infty}},..., t_d^{\pm \frac{1}{p^\infty}} \rangle$ be the perfectoid $\calO_C$-algebra from Example~\ref{ex:PerfectoidBasic}. In this case, one readily checks that $R^\flat \simeq \calO_C^\flat \langle  \underline{t}_1^{\pm \frac{1}{p^\infty}},..., \underline{t}_d^{\pm \frac{1}{p^\infty}} \rangle$ where $\underline{t_i} = (t_i, t_i^{\frac{1}{p}}, t_i^{\frac{1}{p^2}},...) \in R^\flat$ and the completion is $\varpi$-adic (for a pseudouniformizer $\varpi \in \calO_C^\flat$ as in Example~\ref{ex:PerfectoidField}). By Example~\ref{ex:PerfectWitt}, the ring $A_\inf(R)$ is identified with $W(\calO_C^\flat) \langle  v_1^{\pm \frac{1}{p^\infty}},..., v_d^{\pm \frac{1}{p^\infty}} \rangle := \big(W(\calO_C^\flat) [ v_1^{\pm \frac{1}{p^\infty}},..., v_d^{\pm \frac{1}{p^\infty}}]\big)^{\widehat{ }}$ where $v_i$ lifts $\underline{t}_i$ and the completion is $(p,[\varpi])$-adic. Using either description given in Remark~\ref{rmk:Teichmuller}, one checks that $[\underline{t}]_i = v_i$. 
\end{example}

\begin{remark}[The sharp map]
\label{rmk:TeichmullerPerfectoid}
Fix $R \in \Perf_{\calO_C}$. By Remark~\ref{rmk:Teichmuller}, we have a multiplicative map $[-]:R^\flat \to A_\inf(R)$. Using the description $R^\flat \simeq \lim_{x \mapsto x^p} R$ of Lemma~\ref{lem:AinfTheta}, we get a multiplicative map
\[ \sharp:\lim_{x \mapsto x^p} R \xrightarrow{[-]} A_\inf(R) \stackrel{\theta}{\to} R.\]
This map coincides with projection on the first factor, i.e., it sends a sequence $(a_i) \in \lim_{x \mapsto x^p} R$ to $a_0 \in R$. To see this, observe that the claim is certainly true after postcomposition with $R \to R/p$. For the rest, it suffices to argue as in Remark~\ref{rmk:Teichmuller} using the tower $\{R/p^n\}$ instead of the tower $W_n(A)$ in {\em loc. cit.}.
\end{remark}

\begin{remark}[Explicitly describing $\theta$ and $\tilde{\theta}$]
Fix $R \in \Perf_{\calO_C}$. Lemma~\ref{lem:AinfTheta} gives us a map $\theta:W(R^\flat) \to R$. In Remark~\ref{rmk:Teichmuller}, we explained how to elements of $W(R^\flat)$ can be uniquely written as power series $\sum_i [a_i] p^i$ with $a_i \in R^\flat$. As $\theta$ is continuous, to describe it explicitly, it suffices to describe the composition
\[ R^\flat \xrightarrow{[-]} W(R^\flat) \xrightarrow{\theta} R\]
explicitly. As in Remark~\ref{rmk:TeichmullerPerfectoid}, write $b \mapsto b^\sharp$ for this composite map. By examining the proof of Lemma~\ref{lem:FlatPerfectoidAlternate} and using Remark~\ref{rmk:TeichmullerPerfectoid}, this composition is given as follows: given $b = (b_i) \in R^\flat := \lim_{x \mapsto x^p} R/p$, we have $b^\sharp  = \lim_{k \to \infty} \widetilde{b_{i+k}}^{p^k}$, where $\widetilde{b_n} \in R$ denotes some lift of $b_n \in R/p$. This gives us the following formulas:
\[ \theta(\sum_i [a_i] p^i) = \sum_i a_i^\sharp \cdot p^i \quad \text{and} \quad \tilde{\theta}(\sum_i [a_i] p^i) = \sum_i (a_i^{\frac{1}{p}})^\sharp \cdot p^i.\]
Neither of these formulas will be used in the sequel.
\end{remark}

We briefly explain the main consequence of the vanishing of the cotangent complex to the perfectoid theory.

\begin{remark}[Invariance under deformations]
\label{rmk:DefThyAinf}
Let $A$ be one of the four rings appearing in the square \eqref{eq:Ainf} for $R = \calO_C$. The proof of Lemma~\ref{lem:AinfTheta} shows that $A$ is complete for the $(p,\xi)$-adic topology, where $\xi \in A_\inf(\calO_C)$ is a generator for $\ker(\theta)$. Moreover, all maps in the square \eqref{eq:Ainf} are pro-infinitesimal thickenings with respect to this topology. It formally follows from the cotangent complex formalism that the category $\calC_A$ of $(p,\xi)$-adically complete flat $A$-algebras $A'$ with $\widehat{L_{A'/A}} \simeq 0$ (where the completion is $(p,\xi)$-adic) is independent of the ring $A$ that has been chosen; more precisely, for any map $A \to B$ in square \eqref{eq:Ainf} for $R = \calO_C$, the base change functor induces an equivalence $\calC_A \simeq \calC_B$. For $A = \calO_C$, the category $\calC_A$ includes $\Perf_{\calO_C}$ as a full subcategory by Remark~\ref{rmk:PerfectoidCC}, and thus we obtain a description of perfectoid $\calO_C$-algebras in terms of certain $\calO_C^\flat$-algebras, as spelled out in Remark~\ref{rmk:Tilting} below.
\end{remark}

\begin{remark}[The tilting equivalence]
\label{rmk:Tilting}
Let $\calO_C^\flat$ be the valuation ring from Example~\ref{ex:PerfectoidTorusAinf}. One defines the notion of a {\em perfectoid $\calO_C^\flat$-algebra} in exactly the same way as the analogous objects over $\calO_C$, with the element $\varpi \in \calO_C^\flat$ playing the role of $p$; equivalently, these are also just the $\varpi$-adically complete $\varpi$-torsionfree $\calO_C^\flat$-algebras which are perfect rings. If $\Perf_{\calO_C^\flat}$ denotes the resulting category, then the construction $R \mapsto R^\flat$ provides a functor $(-)^\flat:\Perf_{\calO_C} \to \Perf_{\calO_C^\flat}$. Conversely, using Lemma~\ref{lem:AinfTheta}, the construction $S \mapsto W(S) \otimes_{A_\inf(\calO_C),\theta} \calO_C$ yields a functor $(-)^\sharp:\Perf_{\calO_C^\flat} \to \Perf_{\calO_C}$ in the reverse direction; equivalently, one may also define $S^\sharp$ using the equivalence from Remark~\ref{rmk:DefThyAinf}. These functors give mutually inverse equivalences; this is an example of Scholze's {\em tilting equivalence} from \cite[Theorem 5.2]{ScholzeThesis}.
\end{remark}

\subsection{Fontaine's $A_\inf$}
\label{ss:PerfectoidField}

Let $C$ be a complete and algebraically closed nonarchimedean extension of $\Q_p$. Our goal is to discuss the structure of $\calO_C^\flat$ and $A_\inf := A_\inf(\calO_C)$ in this section; using this structure, we will arrive at a better formulation of the main result proven in these notes in \S \ref{ss:BetterMainThm}.

We begin by observing that the valuation on $\calO_C$ pulls back along the multiplicative bijection $\calO_C^\flat \simeq \lim_{x \mapsto x^p} \calO_C$ to yield a valuation on $\calO_C^\flat$ that turns the latter into a valuation ring with the same value group and residue field $k$ as $\calO_C$. Let $\fram^\flat \subset \calO_C^\flat$ denote the maximal ideal, so $\calO_C^\flat/\fram^\flat \simeq k \simeq \calO_C/\fram$. The element $\varpi = \underline{p} \in \calO_C^\flat$ from the proof of Lemma~\ref{lem:AinfTheta} is a pseudouniformizer in $\fram^\flat$. In this situation, the ring $A_\inf := A_\inf(\calO_C)$ comes equipped with a few interesting specializations that are relevant in this paper:
\begin{enumerate}
\item The {\em de Rham} specialization $\theta:A_\inf \to \calO_C$: as explained in Lemma~\ref{lem:AinfTheta}, this is obtained by killing a nonzerodivisor $\xi \in A_\inf$. 
\item The {\em Hodge-Tate} specialization $\tilde{\theta}:A_\inf \to \calO_C$: as $\tilde{\theta} = \theta \circ \phi^{-1}$, this is obtained by killing $\tilde{\xi} := \phi(\xi) \in A_\inf$ for some $\xi$ as in (1).
\item The {\em \'etale} specialization $A_\inf \to W(C^\flat)$: this is obtained by $p$-adically completing $A_\inf[ \frac{1}{[\underline{p}]}]$; equivalently, one applies $W(-)$ to the canonical map $\calO_C^\flat \to C^\flat := \calO_C^\flat[\frac{1}{\underline{p}}]$.
\item The {\em crystalline} specialization $A_\inf \to W(k)$: this is obtained by $p$-adically completing $\colim A_\inf/([\underline{p}]^{\frac{1}{p^n}})$; equivalently, one applies $W(-)$ to the canonical map $\calO_C^\flat \to k$\footnote{For future reference, we let $W(\fram^\flat) \subset A_\inf$ be the kernel of the map $A_\inf \to W(k)$; it is obtained by applying the Witt vector functor $W(-)$ to the nonunital perfect ring $\fram^\flat$, and can thus be defined as the $p$-adic completion of the ideal $\cup_n ([\underline{p}]^{\frac{1}{p^n}}) \subset A_\inf$.}.
\item The {\em modular} specialization\footnote{This notation is non-standard.} $A_\inf \to \calO_C^\flat$: this is obtained by simply setting $p=0$. 
\end{enumerate}
A cartoon of $\Spec(A_\inf)$ depicting the above specializations is given in Figure~\ref{fig:Ainf}.

\begin{remark}[A geometric analogy for $\mathrm{Spec}(A_{\inf})$]
It is instructive to view $\Spec(A_\inf)$ as formally analogous to the two-dimensional regular local ring $\C\llbracket x,y \rrbracket$ of functions on the formal affine plane $\widehat{\A^2}$, with $p$ and $[\underline{p}]$ playing the role of $x$ and $y$; in particular, the modular and crystalline specializations give the two ``axes'' of the plane, while the de Rham specialization cuts out the diagonal (take $\xi = p - [\underline{p}]$). The main philosophical difference is that there is no analog of the inclusion $\C \to \C \llbracket x,y \rrbracket$ of the ground field. More practically, the important difference is that, unlike any noetherian situation, the scheme $\Spec(A_\inf)$ is ``infinitely ramified'' along the closed subset defined by $[\underline{p}] = 0$ as $[\underline{p}] \in A_\inf$ admits $p$-power roots of any order (which is reflected in the fact that the crystalline specialization is not cut out by a single nonzerodivisor). Despite these differences, this analogy can be pushed quite far in some respects; for example, all vector bundles on $\mathrm{Spec}(A_{\inf}) - V(p,[\underline{p}])$ extend uniquely to $\mathrm{Spec}(A_{\inf})$, and are thus free; we refer to \cite[Lemma 4.6]{BMSMainPaper} for a proof, and \cite{KedlayaAinf,ScholzeLectureNotes} for more on this analogy.
\end{remark}

\begin{remark}[$p$-adic period rings]
In addition to the preceding specializations, the following constructions starting with $A_\inf$ give some standard ``period rings'' and thus play an important role in $p$-adic Hodge theory:
\begin{enumerate}
\item The ring $A_\crys$: this is obtained by adjoining divided powers of $\xi$ and $p$-adically completing, i.e., $A_\crys$ is the $p$-adic completion of the subring of $A_\inf[\frac{1}{p}]$ generated by $A_\inf$ and $\frac{\xi^n}{n!}$ for all $n \geq 1$. 
\item The ring $B_{dR}^+$: this discrete valuation ring arises by completing $A_\inf[\frac{1}{p}]$ along the kernel of $\theta[\frac{1}{p}]$.
\end{enumerate}
The $A_\inf \to B_{dR}^+$ and the image of $\Spec(A_\crys) \to \Spec(A_\inf)$ are highlighted in Figure~\ref{fig:Ainf}. In particular, note that the divisor giving the Hodge-Tate specialization meets the latter image, but is not contained in it.
\end{remark}

\begin{warning}
Some of the discussion in this section, including the analogy between $\C \llbracket x,y \rrbracket$ and $A_\inf$, might give the impression that $A_\inf$ has Krull dimension $2$. In fact, the dimension is at least $3$, and we do not know its exact value. To see why it is at least $3$, one uses the chain
\[ (0) \subset \cup_n ([\underline{p}^{\frac{1}{p^n}}]) \subset W(\fram^\flat) \subset (p,W(\fram^\flat))\]
of primes; the issue, again, is one of completions as the third ideal above is the $p$-adic completion of the second one, but not equal to it. 
\end{warning}

\begin{figure}

\includegraphics[width=\textwidth]{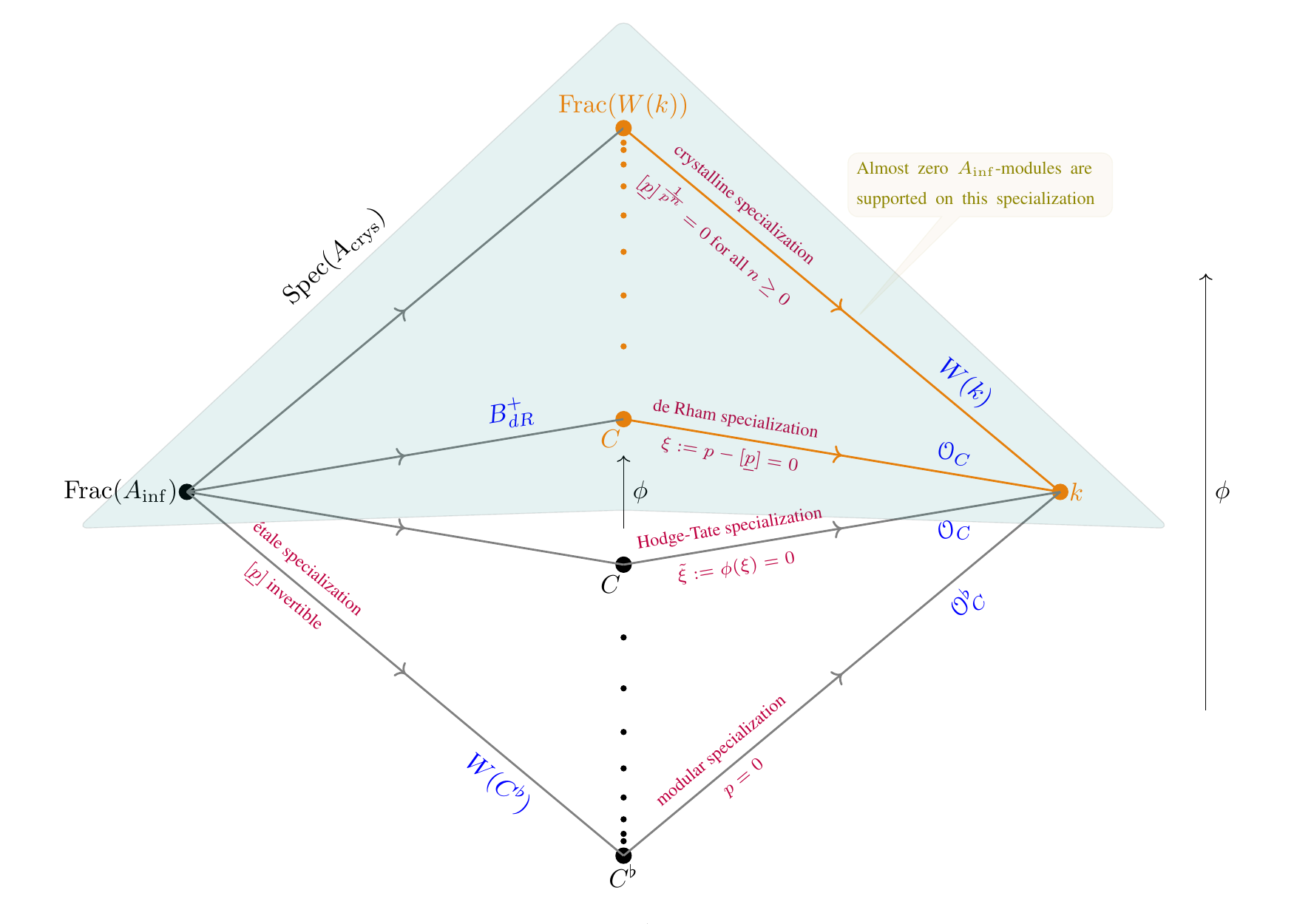}

\captionsetup{singlelinecheck=off}
\caption[A cartoon of $\Spec(A_\inf)$]{A cartoon of $\Spec(A_\inf)$. This depiction of the poset of prime ideals in $A_\inf$ emphasizes certain vertices and edges that are relevant to $p$-adic cohomology theories. 
\begin{itemize}
\item The darkened vertices (labelled `\ding{108}' or `{\color{orange}\ding{108}}') indicate (certain) points of $\Spec(A_\inf)$ and are labelled by the corresponding residue field. 
\item The {\color{gray} gray}/{\color{orange}orange} arrows indicate specializations in the spectrum, while the {\color{blue} blue} label indicates the completed local ring along the specialization. 
\item The locus in $\Spec(A_\inf)$ where almost zero modules live is indicated in an {\color{olive} olive} bubble. 
\item The labels in {\color{purple} purple} match the arrows to one of the previously introduced specializations, and indicate the equations that cut the specializations out in $\mathrm{Spec}(A_\inf)$.
\item The smaller bullets (labelled `{\tiny $\bullet$}' or `{\color{orange}{\tiny $\bullet$}}') down the middle are meant to denote the $\phi^{\Z}$-translates to the two drawn points labelled $C$ (with $\phi^{\Z_{\geq 0}}$ translates of the generic point of the de Rham specialization in {\color{orange} orange}, and the rest in black), and are there to remind the reader that not all points/specialization in $\Spec(A_\inf)$ have been drawn. 
\item The vertices/labels/arrows in {\color{orange} orange} mark the points and specializations that lie in $\Spec(A_\inf/\mu) \subset \Spec(A_\inf)$. 
\item The triangular region covered in {\color{teal} teal} identifies the image of $\Spec(A_\crys) \to \Spec(A_\inf)$. 
\item The arrow labelled $\phi$ on the far right indicates the Frobenius action on $\Spec(A_\inf)$, which fixes the $4$ vertices of the outer diamond in the above picture.
\end{itemize}
} \label{fig:Ainf}
\end{figure}

In the sequel, it will be convenient, especially for calculational purposes, to make a particular choice of a generator $\xi \in \ker(\theta)$ that interacts well with Frobenius. Thus, we fix the following notation for the rest of the paper; all constructions are canonically independent of the choices.

\begin{notation}
\label{not:epsilon}
Fix a compatible sequence $(1,\epsilon_p,\epsilon_{p^2},...)$ of $p$-power roots of $1$ in $\calO_C$ with $\epsilon_{p^n} \in \mu_{p^n}$ and $\epsilon_p \neq 1$; this can also be viewed as a trivialization $\underline{\epsilon}:\Z_p \simeq \Z_p(1)$. Write $\underline{\epsilon} \in \calO_C^\flat$ for the corresponding element in $\calO_C^\flat$, which gives $[\underline{\epsilon}] \in A_\inf$. Note that $[\underline{\epsilon}^a] \in A_\inf$ makes sense for any $a \in \Z[\frac{1}{p}]$ as $\calO_C^\flat$ is perfect and $[-]$ is multiplicative; we will write $\epsilon^a \in \calO_C$ for its image under $\theta$, so $\epsilon^{\frac{m}{p^k}} := (\epsilon_{p^k})^m$ for $m \in \Z$. Set 
\[ \mu = [\underline{\epsilon}] - 1 \in A_\inf\] 
and define 
\[ \xi := \frac{\mu}{\phi^{-1}(\mu)} = \frac{[\underline{\epsilon}] - 1}{[\underline{\epsilon}^{\frac{1}{p}}] - 1} = \sum_{i=0}^{p-1} [\underline{\epsilon}^{\frac{1}{p}}]^i \quad \quad \mathrm{and} \quad \quad \tilde{\xi} := \phi(\xi) = \frac{\phi(\mu)}{\mu} = \frac{[\underline{\epsilon}^p] - 1}{[\underline{\epsilon}] - 1} = \sum_{i=0}^{p-1} [\underline{\epsilon}]^i. \]
The following properties are easily verified, and will be useful in the sequel:
\begin{enumerate}
\item The element $\xi$ (resp. $\tilde{\xi}$) generates the kernel of $\theta$ (resp. $\tilde{\theta}$). 
\item For $a,b \in \Z[\frac{1}{p}]$, if $v_p(a) \leq v_p(b)$, then $[\underline{\epsilon}^a] - 1$ divides $[\underline{\epsilon}^b] - 1$. 
\item For $a \in \Z$, the image of $\frac{[\underline{\epsilon}^a] - 1}{[\underline{\epsilon}] -1}$ in $A_\inf/\mu$ coincides with $a$. In particular, $\tilde{\xi} = p \mod \mu$.
\item For any $n \geq 1$, one has the formula
\[ \mu = \Big(\prod_{i=0}^{n-1} \phi^{-i}(\xi)\Big) \cdot \phi^{-n}(\mu).\]
\item $p$ is a nonzero divisor modulo $\mu$, and vice versa; ditto for $\xi$ and $\tilde{\xi}$ replacing $\mu$.
\item $\tilde{\xi}$ is a nonzero divisor modulo $\mu$, and vice versa.
\item The ideals $(p,\xi)$, $(p,\tilde{\xi})$, $(p,\mu)$ and $(\tilde{\xi},\mu)$ cut out the closed point of $\Spec(A_\inf)$ (set-theoretically), and hence define the same topology on $A_\inf$.
\end{enumerate} 
In terms of Figure~\ref{fig:Ainf}, the formula in (4) above implies that the zero locus $\Spec(A_\inf/\mu) \subset \Spec(A_\inf)$ of $\mu \in A_\inf$ includes the divisors $\Spec(A_\inf/\phi^{-n}(\xi))$ for all $n \geq 0$. 
\end{notation}

\begin{remark}
One may describe the ideal $(\mu) \subset A_\inf$ canonically via the Witt vector functor as follows: as the projection map $R:W(\calO_C) \to \calO_C$ is a pro-infinitesimal thickening modulo all powers of $p$, there is a canonical map $\theta_\infty:A_\inf \to W(\calO_C)$ factoring $\theta$ through $R$ since $\widehat{L_{A_\inf/\Z_p}} \simeq 0$. Then one can show that the kernel of $\theta_\infty$ is exactly $(\mu) \subset A_\inf$ (see \cite[Lemma 3.23]{BMSMainPaper}), giving the promised description. We do not elaborate on this further as we won't need this description (or the Witt vector functor for non-perfect rings) in these notes. 
\end{remark}

\subsection{A better formulation of the main theorem}
\label{ss:BetterMainThm}

Using the structure $A_\inf$ explained in \S \ref{ss:PerfectoidField}, one can give a better formulation of the main theorem proven in these notes:

\begin{theorem}
\label{thm:MainThmPrecise}
Let $C$ be spherically complete\footnote{Recall that a nonarchimedean field is spherically complete if all descending sequences of discs have a non-empty intersection. Every nonarchimedean field has a unique spherical completion. We refer to \cite[\S 1.5]{KedlayaPDE} for the basic theory of such fields. The only consequence of the spherical completeness of $C$ relevant to these notes is a vanishing theorem, captured in Lemma~\ref{lem:SphCompleteVanishing}. As explained in Remark~\ref{rmk:Deficiencies}, this hypothesis is not an essential one.}, and fix a proper smooth formal scheme $\frakX$ over $\calO_C$. Then one can functorially attach to $\frakX$ a perfect complex $R\Gamma_A(\frakX) \in D_\perf(A_\inf)$. This complex has the following specializations:
\begin{enumerate}
\item de Rham: there is a canonical isomorphism
\[ R\Gamma_A(\frakX) \otimes^L_{A_\inf,\theta} \calO_C \simeq R\Gamma_{dR}(\frakX/\calO_C).\]
\item \'Etale: there is a canonical isomorphism 
\[ R\Gamma_A(\frakX) \otimes^L_{A_\inf} W(C^\flat) \simeq R\Gamma_{\et}(\frakX_C, \Z_p) \otimes^L_{\mathbf{Z}_p} W(C^\flat).\]
\end{enumerate}
\end{theorem}

Using the output of this construction, one can formally prove Theorem~\ref{thm:MainThm} as follows:

\begin{proof}[Proof of Theorem~\ref{thm:MainThm}]
Fix a proper smooth formal scheme $\frakX/\calO_C$. Assume first that $C$ is spherically complete. Theorem~\ref{thm:MainThmPrecise} then gives us a complex $R\Gamma_A(\frakX)$ whose modular specialization $R\Gamma_{\calO_C^\flat}(\frakX) := R\Gamma_A(\frakX) \otimes^L_{A_\inf} \calO_C^\flat$ is a perfect $\calO_C^\flat$-complex whose fibres are identified as follows:
\[R\Gamma_{\calO_C^\flat}(\frakX) \otimes^L_{\calO_C^\flat} C^\flat \simeq R\Gamma_{\et}(\frakX_C, \F_p) \otimes^L_{\F_p} C^\flat \quad \mathrm{and} \quad R\Gamma_{\calO_C^\flat}(\frakX) \otimes^L_{\calO_C^\flat} k \simeq R\Gamma_{dR}(\frakX_k/k).\]
In particular, by semicontinuity for ranks of the cohomology groups of the fibres of a perfect complex, we get an inequality
\[ \dim_{\F_p} H^i_\et(\frakX_C, \F_p) \leq \dim_k H^i_{dR}(\frakX_k/k),\]
which proves Theorem~\ref{thm:MainThm} when $C$ is spherically complete. The general case of Theorem~\ref{thm:MainThm} then follows immediately as we may replace $C$ by a larger spherically closed algebraically closed extension (see \cite[Chapter 3]{Robertpadicbook} for a construction of the desired field extension, and \cite[0.3.2]{Huber} for insensitivity of \'etale cohomology to such extensions in the rigid-analytic setting).
\end{proof}

\begin{remark}[Deficiencies of Theorem~\ref{thm:MainThmPrecise}]
\label{rmk:Deficiencies}
We briefly comment on some of some defects of Theorem~\ref{thm:MainThmPrecise}, when compared to \cite[Theorem 1.8]{BMSMainPaper}. First, the assumption that $C$ is spherically complete is not important for the theorem, cf. Remark~\ref{rmk:RGammaANoAlmost}. More seriously, there is also a good description of the crystalline specialization: by \cite[Theorem 1.8 (iii)]{BMSMainPaper}, one has a canonical isomorphism
\[ R\Gamma_A(\frakX) \otimes^L_{A_\inf} W(k) \simeq R\Gamma_{\crys}(\frakX_k/W(k)),\]
compatible with Frobenius. The method used in these notes relies on almost mathematics over $A_\inf$, which necessitates systematically ignoring $A_\inf$-modules that come from $W(k)$-modules via restriction of scalars. Thus, it does not seem to be easy to deduce the crystalline comparison using the argument given here (at least without developing some relative de Rham-Witt theory).
\end{remark}

\begin{remark}[The absolute crystalline comparison and some open questions]
In the context of the cohomology theory $R\Gamma_A(\frakX)$ from Theorem~\ref{thm:MainThmPrecise}, the region covered by $\Spec(A_\crys)$ in Figure~\ref{fig:Ainf} is the ``largest'' region over which the theory $R\Gamma_A(\frakX)$ can be constructed using techniques of crystalline cohomology. More precisely, the absolute crystalline cohomology of $\frakX \otimes_{\calO_C} \calO_C/p$ enters the mix via a natural isomorphism
 \begin{equation}
 \label{eq:CrysComp}
 R\Gamma_A(\frakX) \otimes^L_{A_\inf} A_\crys \simeq R\Gamma_\crys( (\frakX \otimes_{\calO_C} \calO_C/p)/\Z_p)
 \end{equation}
 by \cite[Theorem 1.8 (iv)]{BMSMainPaper}. Moreover, the construction of $R\Gamma_A(\frakX)$ shows that 
 \begin{equation}
 \label{eq:EtaleComp}
R\Gamma_A(\frakX)[\frac{1}{\mu}] \simeq R\Gamma_\et(\frakX_C, \Z_p) \otimes_{\Z_p} A_\inf[\frac{1}{\mu}],
\end{equation}
where $\mu \in A_\inf$, as in Notation~\ref{not:epsilon}, is an element whose zero locus in $\Spec(A_\inf)$ can be viewed heuristically as the union of all $\Spec(A_\inf/\phi^{-n}(\xi))$ for $n \geq 0$ with $\Spec(W(k))$ (see Figure~\ref{fig:Ainf} for a better picture of $\Spec(A_\inf/\mu)$, and Proposition~\ref{prop:RGammaAOmega} (iii) for a proof of \eqref{eq:EtaleComp}). Combining these gives a large subset of $\Spec(A_\inf)$ where the fibres of $R\Gamma_A(\frakX)$ can be identified classically. However, this glueing process does not cover $\Spec(A_\inf)$. In particular, it does not help describe the Hodge-Tate or the modular specialization in classical terms,  leading to the following questions:
 \begin{enumerate}
 \item Is there a crystalline construction of the Hodge-Tate specialization $R\Gamma_A(\frakX) \otimes^L_{A_\inf,\tilde{\theta}} \calO_C$? For example, does it arise as $R\Gamma(\frakX,-)$ applied to a natural sheaf of commutative dg algebras in $D_{qc}(\frakX)$? The construction (via the complex $\widetilde{\Omega}_\frakX$ from \S \ref{sec:TildeOmega}) shows that the answer is ``yes'' if one works with $E_\infty$-algebras instead, so this question is really asking if the $E_\infty$-algebra $\widetilde{\Omega}_{\frakX} \in D_{qc}(\frakX)$ admits a ``strict'' model. Note that \cite[Remark 7.8]{BMSMainPaper} provides a non-trivial obstruction to representing $A\Omega_\frakX/p$ by a strict object; however, this vanishes on $\widetilde{\Omega}_{\frakX}/p$ as $(\epsilon_p - 1)^p = 0 \in \calO_C/p$ for $\epsilon_p \in \mu_p$.
 
  \item Can the modular specialization $R\Gamma_A(\frakX) \otimes^L_{A_\inf} \calO_C^\flat$ be constructed directly using the generic fibre? More concretely, is there an example of two proper smooth (formal) schemes $\frakX/\calO_C$ and $\frakY/\calO_C$ such that $\frakX_C \simeq \frakY_C$, and yet $\frakX_k$ and $\frakY_k$ have different de Rham cohomologies? At the level of individual cohomology groups and under some torsionfreeness assumptions on $H^*_\crys(\frakX_k)$, such examples can not exist by \cite[Theorem 1.4]{BMSMainPaper}. On the other hand, at the level of Artin stacks, such examples do exist.\footnote{Let $G = \Z/p^2$, viewed as a constant finite flat group scheme over $\calO_C$, let $E/\calO_C$ be an elliptic curve with supersingular reduction, and let $H \subset E$ be the finite $\calO_C$-flat subgroup generated generated by a point of exact order $p^2$ in $E(C)$. Then $H_C \simeq \Z/p^2$, and $H_k \simeq E[p]$. Then the stacks $BG$ and $BH$  have the same generic fibre, but different de Rham cohomologies for the special fibre \cite[\S 2.2]{BMSMainPaper}.} However, it is not clear to the author how to approximate these examples using smooth proper (formal) schemes.
 \end{enumerate}
\end{remark}

\begin{remark}[The Hodge-Tate spectral sequence]
\label{rmk:HodgeTate}
The de Rham, \'etale, and crystalline specializations acquire their names thanks to the comparison isomorphisms explained above. Likewise, the Hodge-Tate specialization $\tilde{\theta}:A_\inf \to \calO_C$ is closely connected to the Hodge-Tate filtration from \cite[\S 3.3]{ScholzeSurvey}: for any smooth formal scheme $\frakX/\calO_C$, thanks to Proposition~\ref{prop:RGammaAOmega} (4), there is an $E_2$ spectral sequence
\[ E_2^{i,j}: H^i(\frakX, \Omega^j_{\frakX/\calO_C}\{-j\}) \Rightarrow H^{i+j}(R\Gamma_A(\frakX) \otimes^L_{A_\inf,\tilde{\theta}} \calO_C)\]
describing the Hodge-Tate specialization of $R\Gamma_A(\frakX)$. (Here the twist $\{-j\}$ is an integral modification of the Tate twist, see Definition~\ref{def:BKTwist}.) When $\frakX/\calO_C$ is also proper, the target is identified with $H^{i+j}_{\et}(\frakX_C, \Z_p) \otimes C$ after inverting $p$ by formula \eqref{eq:EtaleComp} above: the generic point of the Hodge-Tate specialization does not lie in $\Spec(A_\inf/\mu)$ in Figure~\ref{fig:Ainf}. Thus, after inverting $p$, the preceding spectral sequence recovers the Hodge-Tate one from \cite[Theorem 3.20]{ScholzeSurvey}; we refer the interested reader to \cite{AbbesGrosHodgeTate} for a detailed discussion of this spectral sequence.
\end{remark}

\section{Some almost mathematics}
\label{sec:Almost}

Fix a complete and algebraically closed nonarchimedean field $C$. The basic philosophy informing Faltings' approach \cite{FaltingsJAMS,FaltingsRamified,FaltingsAlmostEtale} to $p$-adic Hodge theory is to do commutative algebra over $\calO_C$ up to a ``very small error''. The latter phrase is codified in the following definition:

\begin{definition}
\label{def:AlmostIsomOC}
An $\calO_C$-module $M$ is {\em almost zero} if $\fram  \cdot M = 0$. A map $f:K \to L$ in $D(\calO_C)$ is an {\em almost isomorphism} if the cohomology groups of the cone of $f$ are almost zero.
\end{definition}

\begin{remark}[The category of almost modules]
\label{rmk:AlmostMathDiscrete}
The preceding definition is meaningful because $\fram^2 = \fram$ (since the value group of $C$ is divisible). Indeed, this formula implies\footnote{If $C$ were discretely valued, then the closure of $\fram$-torsion modules under extensions and colimits would be all torsion $\calO_C$-modules, and the corresponding quotient category would be just that of $C$-vector spaces.} that almost zero modules form an abelian subcategory $\Sigma \subset \Mod_{\calO_C}$ of all $\calO_C$-modules that is closed under subquotients and extensions. Thus, there is a good way to ignore almost zero modules: one simply works in the Serre quotient $\Mod_{\calO_C}^a := \Mod_{\calO_C}/\Sigma$. This category is abelian, and inherits a $\otimes$-structure from $\Mod_{\calO_C}$; thus, one can make sense of basic notions of commutative algebra and algebraic geometry in $\Mod_{\calO_C}^a$ (see Gabber-Ramero). There are obvious (exact $\otimes$-) functors $\Mod_{\calO_C} \to \Mod_{\calO_C}^a \to \Mod_C$ which sandwich the almost integral objects (i.e., $\Mod_{\calO_C}^a$) in between integral objects (i.e., $\Mod_{\calO_C}$) and rational ones (i.e., $\Mod_C$);  in practice, the almost integral theory turns out to be much closer to the integral one than the rational one. In fact, one goal of these notes is to highlight an approach for passing from almost integral statements to integral ones, building on Remark~\ref{rmk:SphCompletePerfectComplex} below. 
\end{remark}

%
%
%

The book \cite{GabberRamero} develops almost analogs of various constructions in commutative algebra and algebraic geometry. We do not need or use these notions in a serious way in these notes. However, it will be convenient to develop some language to discuss the derived category version of almost mathematics, summarized in the following construction:

\begin{construction}[The almost derived category over $\calO_C$]
\label{cons:AlmostOC}
Since $\fram \otimes^L_{\calO_C} \fram \simeq \fram^2 = \fram$, the restriction of scalars functor $D(k) \to D(\calO_C)$ is fully faithful, and its essential image if $D_\Sigma(\calO_C)$, i.e., those $K \in D(\calO_C)$ with $H^i(K)$ being almost zero; let $D(\calO_C)^a := D(\calO_C)/D(k)$ for the Verdier quotient.\footnote{By \cite[2.4.9]{GabberRamero}, one has $D(\calO_C)^a \simeq D(\Mod_{\calO_C}^a)$.} Thus, the quotient functor $D(\calO_C) \to D(\calO_C)^a$ is the localization of $D(\calO_C)$ along almost isomorphisms; write $K^a \in D(\calO_C)^a$ for the image of $K \in D(\calO_C)$. This functor has a left-adjoint $K^a \mapsto K^a_! := \fram \otimes_{\calO_C} K$ and a right adjoint $K^a \mapsto K^a_*  := R\Hom_{\calO_C}(\fram, K)$; here we are implicitly asserting that the given formula only depends on $K$ up to almost isomorphisms. One checks that both adjoints are fully faithful. Note that quotient functor as well as the two adjoints preserve $p$-adically complete objects as any $K \in D_{\Sigma}(\calO_C)$ is  $p$-adically complete. We sometimes abuse notation and write $K_*$ instead of $(K^a)_*$ for $K \in D(\calO_C)$, and similarly for $K_!$.
\end{construction}

Recall that $C$ is said to be {\em spherically complete} if any nested sequence of closed discs in $C$ has a non-empty intersection. The main algebraic fact about almost mathematics relative to $\calO_C$ that we will need is:

\begin{lemma}
\label{lem:SphCompleteVanishing}
Let $C$ be spherically complete. For any perfect complex $K \in D(\calO_C)$, we have $K \simeq (K^a)_\ast$.
\end{lemma}
\begin{proof}
By choosing a resolution of $K$, we may assume $K = \calO_C$. Thus, we must check that $R\Hom(\fram, \calO_C) =: (\calO_C^a)_\ast$ is simply $\calO_C[0]$. We may write $\fram = \cup_n \fram_n$ with each $\fram_n$ being principal. We have to show that every map $\fram \to \calO_C$ extends to a map $\calO_C \to \calO_C$, and that $R^1 \lim (\fram_n^\vee) = 0$, where $\fram_n^\vee := \Hom(\fram_n, \calO_C)$. For the first, note that a map $\fram \to \calO_C$ is given by an element $x \in C$ such that $\epsilon \cdot x \in \calO_C$ for every $\epsilon \in \fram$. But then $v(x) + v(\epsilon) \geq 0$ for each $\epsilon \in \fram$, and thus $v(x) \geq 0$, so $x \in \calO_C$. For the second, we first note that the canonical map $\alpha:C \to \lim C/\fram_n^\vee$ is surjective by spherical completeness of $C$\footnote{Indeed, an element $\{\overline{a_n} \} \in \lim C/\fram_n^\vee$ determines a descending sequence $\{a_n + \fram_n^\vee\}$ of open discs in $C$, where $a_n \in C$ is any lift of $\overline{a_n} \in C/\fram_n^\vee$. By spherical completeness, there is some $a \in C$ such that $a \in a_n + \fram_n^\vee$ for all $n$; but then $a \in C$ maps to $\{\overline{a_n} \} \in \lim C/\fram_n^\vee$ under the above map, giving surjectivity.}. The long exact sequence for $R^i \lim$ applied to the short exact sequence 
\[ 0 \to \{\fram_n^\vee\} \to \{C\} \to \{C/\fram_n^\vee\} \to 0\] 
of projective systems then gives the desired $\lim^1$ vanishing. 
\end{proof}

\begin{remark}[Spectral sequences]
\label{rmk:SphCompletePerfectComplex}
Let $C$ be spherically complete. Lemma~\ref{lem:SphCompleteVanishing} can be phrased as follows: the functor $K \mapsto K^a$ gives an equivalence between $D_{\perf}(\calO_C)$ and the full subcategory $D_{\perf}(\calO_C)^a \subset D(\calO_C)^a$ spanned by those bounded complexes $K \in D(\calO_C)^a$ such that each $H^i(K)_*$ is finitely presented (or, equivalently, perfect). From henceforth, a perfect complex in $D(\calO_C)^a$ refers an object of $D_{\perf}(\calO_C)^a$. A formal consequence of this equivalence is the following abstract statement: the functor $(-)_*$ gives an equivalence between the category of first quadrant $E_2$-spectral sequences in $\Mod_{\calO_C}^a$ whose terms are perfect in the preceding sense and the analogous category defined using $\Mod_{\calO_C}$.
\end{remark}

\begin{remark}
Some assumption on $C$ beyond completeness is necessary for Lemma~\ref{lem:SphCompleteVanishing}. Indeed, for $C = \widehat{\overline{\Q_p}}$, we have $\Ext^2_{\calO_C}(k, \calO_C) \neq 0$, and hence $\calO_C \to (\calO_C^a)_*$ is not an isomorphism.
\end{remark}

Earlier, we introduced the ideal $W(\fram^\flat) = \ker(A_\inf \to W(k))$ satisfying $W(\fram^\flat) \cdot \calO_C = \fram$. For applications, it is convenient to extend the basic notions of almost mathematics from the pair $(\calO_C,\fram)$ to the pair $(A_\inf,W(\fram^\flat))$:

\begin{definition}
An $A_\inf$-module $M$ is {\em almost zero} if $W(\fram^\flat) \cdot M = 0$. A map $f:K \to L$ in $D(A_\inf)$ is an {\em almost isomorphism} if the cohomology groups of the cone of $f$ are almost zero.
\end{definition}

The preceding definition comes with a warning: we do not know if $W(\fram^\flat)^2 \neq W(\fram^\flat)$ as ideals in $A_\inf$ (due to issues of completion, see \cite[Remark 1.4]{KedlayaAinf}), so almost zero modules need not form a Serre subcategory of all $A_\inf$-modules, and thus the usual formalism of almost mathematics (as discussed in Remark~\ref{rmk:AlmostMathDiscrete}) does not automatically apply. Nevertheless, in these notes, we will work only with complexes of $A_\inf$-modules which are $p$-adically complete; in fact, we will typically only encounter $(p,\mu)$-adically complete objects. In this context, the preceding subtlety can be skirted, and the formalism of almost mathematics can be salvaged as follows. 

\begin{construction}[The almost derived category over $A_\inf$]
\label{cons:AlmostAinf}
Let $D_{comp}(A_\inf) \subset D(A_\inf)$ be the full subcategory of all $p$-adically complete complexes. Since $W(\fram^\flat) \widehat{\otimes}^L_{A_\inf} W(\fram^\flat) \simeq W(\fram^\flat)$, the restriction of scalars functor $D_{comp}(W(k)) \to D_{comp}(A_\inf)$ is fully faithful, and its essential image comprises those $K \in D_{comp}(A_\inf)$ with each $H^i(K)$ being almost zero. Let $D_{comp}(A_\inf)^a := D_{comp}(A_\inf)/D_{comp}(W(k))$ be the Verdier quotient, and write $K \mapsto K^a$ for the quotient map $D_{comp}(A_\inf) \to D_{comp}(A_\inf)^a$. This functor has a left-adjoint defined by $K^a \mapsto K^a_! := K \widehat{\otimes}_{A_\inf} W(\fram^\flat)$, and a right-adjoint defined by $K^a \mapsto K^a_* := R\Hom_{A_\inf}(W(\fram^\flat), K)$. As in Construction~\ref{cons:AlmostOC}, we are implicitly asserting that the given formulas only depend on $K$ up to almost isomorphism, and one checks that both these adjoints are fully faithful. We sometimes abuse notation and write $K_*$ instead of $(K^a)_*$ for $K \in D_{comp}(A_\inf)$, and similarly for $K_!$.
\end{construction}

\begin{remark}
Constructions~\ref{cons:AlmostOC} and \ref{cons:AlmostAinf} are compatible in an evident sense: if $D_{comp}(\calO_C) \subset D(\calO_C)$ denotes the full subcategory of $p$-adically complete complexes, then, for any $K \in D_{comp}(\calO_C)$, the formation of $(K^a)_*$ commutes with the forgetful functor $D_{comp}(\calO_C) \to D_{comp}(A_\inf)$, and similarly for $(K^a)_!$. 
\end{remark}

The $A_\inf$-analogue of Lemma~\ref{lem:SphCompleteVanishing} is the next lemma, which will be quite useful in \S \ref{sec:globalAOmega} for converting certain almost isomorphisms into true isomorphisms:

\begin{lemma}
\label{lem:SphCompleteVanishingAinf}
Let $C$ be spherically complete. If $K \in D(A_\inf)$ is perfect, then $K \simeq (K^a)_*$.
\end{lemma}
\begin{proof}
We may assume $K = A_\inf$. Thus, we must check that $R\Hom_{A_\inf}(W(\fram^\flat), A_\inf) \simeq A_\inf$. As both sides are $\xi$-adically complete, it suffices to check the same after applying $(-) \otimes^L_{A_\inf} A_\inf/\xi$. Now adjunction gives
\[ R\Hom_{A_\inf}(W(\fram^\flat), A_\inf) \otimes^L_{A_\inf} A_{\inf}/\xi \simeq R\Hom_{A_\inf}(W(\fram^\flat), \calO_C) \simeq R\Hom_{\calO_C}(W(\fram^\flat) \otimes^L_{A_\inf} \calO_C, \calO_C).\]
Moreover, $W(\fram^\flat) \otimes^L_{A_\inf} \calO_C$ is $p$-adically complete (as $\calO_C$ is a perfect $A_\inf$-complex), and thus identifies with $W(\fram^\flat) \widehat{\otimes}_{A_\inf} \calO_C \simeq \fram$. It is thus enough to check $R\Hom_{\calO_C}(\fram, \calO_C) \simeq \calO_C$, which follows from Lemma~\ref{lem:SphCompleteVanishing}.
\end{proof}

\section{Framed algebras and input from perfectoid geometry}
\label{sec:FramedAlgebras}

In this section, we introduce the main calculational tools in the perfectoid approach to $p$-adic Hodge theory: $p$-adic formal tori over $\calO_C$ and their \'etale covers. The general strategy of proving results\footnote{What follows is an oversimplified overview of the steps involved in proving $p$-adic comparison theorems in the good reduction case via this approach; we refer to the original papers of Faltings and Scholze for precise statements, and \S \ref{sec:TildeOmega} for an instance of this strategy in action.} is: 
\begin{enumerate}[(a)]
\item Formulate an appropriate global statement.
\item Check the statement for tori. 
\item Check the statement for \'etale covers of tori by devissage.
\item Pass to the global case by glueing the output of (c) for a suitable Zariski cover.
\end{enumerate}
In practice, a key technical hurdle is often step (b), and this is typically circumvented (or, rather, reduced to an explicit calculation) using the almost purity theorem \cite[\S 2b]{FaltingsAlmostEtale}, \cite[Theorem 7.9]{ScholzeThesis}. 

Thus, we begin in \S \ref{ss:FramedAlgebras} by introducing some notation surrounding \'etale covers of tori in the $p$-adic setting; thanks to a result of Kedlaya \cite{KedlayaAffine} in characteristic $p$ geometry, we can actually get by with {\em finite} \'etale covers, which will simplify some arguments later (but is not essential to the method). Having introduced this notation, we explain in \S \ref{ss:APT} how the almost purity theorem gives an explicit description of certain ``nearby cycles'' complexes of interest in $p$-adic geometry; we adopt here the framework of adic spaces as advocated by Scholze \cite{ScholzePAdicHodge} in lieu of the older approach via the Faltings topos (see \cite{FaltingsAlmostEtale,AbbesGrosTopos}), and also recall the fundamental theorem (see Theorem~\ref{thm:PrimitiveCompThm}) identifying \'etale cohomology in terms of certain nearby cycles complexes in the global case.

\subsection{Framed algebras}
\label{ss:FramedAlgebras}

Our goal is to study formal schemes over $\calO_C$ in relation to their geometric fibres. The relevant class of formal schemes is defined as follows:

\begin{definition}
Let $R$ be a $p$-adically complete and flat $\calO_C$-algebra, viewed as a topological $\mathcal{O}_C$-algebra via the $p$-adic topology. We say that $R$ is {\em formally smooth} over $\calO_C$ if $R/p$ is smooth over $\calO_C/p$. A formal scheme over $\Spf(\calO_C)$ is {\em smooth} if it is Zariski locally of the form $\Spf(R)$ for a formally smooth $\calO_C$-algebra $R$. 
\end{definition}

By a theorem of Elkik \cite{Elkik}, any formally smooth $\mathcal{O}_C$-algebra $R$ as defined above arises as the $p$-adic completion of a smooth $\mathcal{O}_C$-algebra. The most important example of a smooth formal $\calO_C$-scheme for our purposes is:

\begin{example}
The formal torus $\mathbb{T}^d_{\calO_C}$ is given by $\Spf(\calO_C \langle t_1^{\pm 1},...,t_d^{\pm 1} \rangle)$, where the ring is given the $p$-adic topology.
\end{example}

Our strategy is to study smooth formal $\calO_C$-schemes by relating them to formal tori via \'etale co-ordinates, and also studying their deformations along the pro-infinitesimal thickening $A_\inf \stackrel{\theta}{\to} \calO_C$. Thus, it is quite convenient to introduce some notation concerning formal tori and their deformations to such thickenings:

\begin{notation}
\label{not:Torus}
Let $A \in \{\calO_C, A_\inf, A_\inf/\mu\}$; each such an $A$ is viewed as an $A_\inf$-algebra (in the obvious way for $A_\inf$ and $A_\inf/\mu$, and via $\theta$ for $\calO_C$), and is equipped with the $(p,\mu)$-adic topology (and note that $\mu = 0$ for $A \in \{\calO_C,A_\inf/\mu\}$). Write $P^d_A = A \langle t_1^{\pm 1},..., t_d^{\pm 1} \rangle$ for the co-ordinate ring of the formal torus $\mathbb{T}^d_A$ of dimension $d$ over $A$; we also write $P^d = P^d_{\calO_C}$. Taking inverse limits along the multiplication by $p$ map on $\mathbb{T}^d_A$ gives a formal scheme $\mathbb{T}^d_{A,\infty}$ whose co-ordinate ring is denoted $P^d_{A,\infty}$; again, we drop the subscript ``$A$'' when $A = \calO_C$. As an $A$-algebra, $P^d_{A,\infty}$ is described by
\begin{equation}
\label{eq:AlgStructureStdPerfectoid}
P^d_{A,\infty} = A \langle t_1^{\pm \frac{1}{p^\infty}},..., t_d^{\pm \frac{1}{p^\infty}} \rangle.
\end{equation}
In particular, the structure of $P^d_{A,\infty}$ as an $A$-module is given by
\begin{equation}
\label{eq:ModuleStructureStdPerfectoid} 
P^d_{A,\infty} \simeq \widehat{\bigoplus_{(a_1,...,a_d) \in \Z[\frac{1}{p}]^d}} A \cdot \prod_{i=1}^d t_i^{a_i},
\end{equation}
where the completion is $(p,\mu)$-adic. The submodule $P^d_A \subset P^d_{A,\infty}$ is spanned by the (completed) summands indexed by $\Z^d \subset \Z[\frac{1}{p}]^d$. Note that the ring $P^d_\infty$ is a perfectoid $\calO_C$-algebra (see Example~\ref{ex:PerfectoidBasic}).
\end{notation}

\begin{warning}
\label{warning:tiltedcoordinates}
In the notation above, by Example~\ref{ex:PerfectoidTorusAinf}, one has the following formula: 
\begin{equation}
\label{eq:StructureStdPerfectoidLift}
P^d_{A_\inf,\infty} \simeq A_\inf(P^d_\infty).
\end{equation}
 However, this is a bit misleading in terms of functoriality with respect to the choice of co-ordinates. Indeed, the canonical identification proven in Example~\ref{ex:PerfectoidTorusAinf} is $A_\inf(P^d_\infty) = A_\inf \langle v_1^{\pm \frac{1}{p^\infty}},...,v_d^{\pm \frac{1}{p^\infty}} \rangle$ with $v_i = [\underline{t_i}]$ being the Teichmuller lifts of the tilted co-ordinates $\underline{t_i}$ on $(P^d_\infty)^\flat$. In particular, one must heed this warning when transporting structure from $P^d_\infty$ to $P^d_{A_\inf,\infty}$ via \eqref{eq:StructureStdPerfectoidLift} and the functoriality of the $A_\inf(-)$ operation. For example, if a group $G$ acts $P^d$-linearly on $P^d_\infty$, then one obtains by functoriality a $G$-action on $P^d_{A_\inf,\infty}$, but this action need not be trivial on the subring $P^d_{A_\inf} \subset P^d_{A_\inf,\infty}$; an explicit example of this phenomenon arises in Construction~\ref{cons:DeltaActionTorus} below. 
\end{warning}

Recall that in Notation~\ref{not:epsilon} we have a fixed a compatible sequence of $p$-power roots of $1$, viewed as an element $\underline{\epsilon} \in \calO_C^\flat$. This choice enables us to explicitly describe the canonical Galois action present on the perfectoid covers of tori introduced above:

\begin{construction}[The covering group action on a perfectoid torus]
\label{cons:DeltaActionTorus}
The multiplication by $p^k$ map $\mathbb{T}^d_{\calO_C} \to \mathbb{T}^d_{\calO_C}$ is an fppf $(\mu_{p^k})^{\oplus d}$-torsor. Passing to the inverse limit over all $k$, the canonical map $\mathbb{T}^d_{\calO_C,\infty} \to \mathbb{T}^d_{\calO_C}$ is a torsor for the pro-(finite group scheme) $\Z_p(1)^{\oplus d}$. In particular, there is a continuous action of the profinite group $\Z_p(1)^{\oplus d}(\calO_C)$ on $P^d_{\infty}$. Using the identification $\Z_p(1)^{\oplus d}(\calO_C) \stackrel{[\underline{\epsilon}]}{\simeq} \Z_p^{\oplus d} =: \Delta^d$ coming from the choice of $[\underline{\epsilon}]$, this induces a $P^d$-linear $\Delta^d$-action on $P^d_\infty$. Explicitly, this action respects the decomposition in \eqref{eq:ModuleStructureStdPerfectoid} for $A = \calO_C$, and the $j$-th basis vector $e_j \in \Z_p^{\oplus d}$ acts on the monomial $\prod_{i=1}^d t_i^{a_i}$ by multiplying it by $\epsilon^{a_i}$. Formula \eqref{eq:StructureStdPerfectoidLift} then induces a $\Delta^d$-action on $P^d_{A_\inf,\infty}$ (after taking Warning \ref{warning:tiltedcoordinates} into account), which can be described explicitly as follows: the decomposition in \eqref{eq:ModuleStructureStdPerfectoid} (for $A = A_\inf$) is $\Delta^d$-equivariant, and the induced action of the $j$-th basis vector $e_j \in \Z_p^{\oplus d}$ on the summand $A_\inf \cdot \prod_{i=1}^d t_i^{a_i}$ is given by multiplication by $[\underline{\epsilon}^{a_j}]$. This action preserves the subring  $P^d_{A_\inf} \subset P^d_{A_\inf,\infty}$, and the induced $\Delta^d$-action on $P^d_{A_\inf}$ is trivial modulo $\mu$ since $[\underline{\epsilon}^{a_i}] = 1$ modulo $\mu$ provided $a_i \in \Z$. Thus, by passage to the quotient, this induces a $P^d_{A_\inf/\mu}$-linear action on $P^d_{A_\inf/\mu}$ lifting the $P^d$-linear action on $P^d_\infty$. 
\end{construction}

The group cohomology of the Galois group $\Delta^d$ acting on co-ordinate rings of the perfectoid tori can be calculated in terms of explicit Koszul complexes, and forms the basic ingredient of many calculations to follow:

\begin{lemma}
\label{lem:GroupCohKoszulCohPerfectoidTorus}
Fix some $A$ as in Notation~\ref{not:Torus}. 
\begin{enumerate}[(1)]
\item We have a canonical isomorphism
\[ R\Gamma_{conts}(\Delta^d, P^d_{A,\infty}) \simeq  \widehat{\bigoplus_{(a_1,...,a_d) \in \Z[\frac{1}{p}]^d}} K(A; [\underline{\epsilon}^{a_1}] - 1,...., [\underline{\epsilon}^{a_d}] - 1),\]
where the completion is $(p,\mu)$-adic. 
\item When $A = A_\inf/\mu$ or $A = \calO_C$, this decomposes further as
\begin{equation}
\label{eq:RGammaPerfectoidModmu}
R\Gamma_{conts}(\Delta^d, P^d_{A,\infty}) \simeq  \Big(R\Gamma_{conts}(\Delta^d,\Z_p) \otimes_{\Z_p} P^d_A\Big) \oplus \widehat{\bigoplus_{(a_1,...,a_d) \in \Z[\frac{1}{p}]^d-\Z^d}} K(A; [\underline{\epsilon}^{a_1}] - 1,...., [\underline{\epsilon}^{a_d}] - 1).
\end{equation}
\item In the decomposition in equation \eqref{eq:RGammaPerfectoidModmu}, each homology group of the summand $K(A; [\underline{\epsilon}^{a_1}] - 1,...., [\underline{\epsilon}^{a_d}] - 1)$ appearing on the right is a finite direct sum of  copies of $A/([\underline{\epsilon}^a]-1)$ for some $a \in \Z[\frac{1}{p}] - \Z$.
\item The second summand on the right in equation \eqref{eq:RGammaPerfectoidModmu} is annihilated by $[\underline{\epsilon}^{\frac{1}{p}}] - 1$.
\end{enumerate}
\end{lemma}
\begin{proof}
(1) follows from equation~\eqref{eq:ModuleStructureStdPerfectoid}, the compatibility of $R\Gamma_{conts}(\Delta^d,-)$ with completed direct sums, and the standard fact that the group cohomology of topologically free groups is computed by a Koszul complex. For (2), breaking the index set in (1) into its integral and nonintegral pieces shows
\[ R\Gamma_{conts}(\Delta^d, P^d_{A,\infty}) \simeq R\Gamma_{conts}(\Delta^d,P^d_A)  \oplus \widehat{\bigoplus_{(a_1,...,a_d) \in \Z[\frac{1}{p}]^d-\Z^d}} K(A; [\underline{\epsilon}^{a_1}] - 1,...., [\underline{\epsilon}^{a_d}] - 1).\]
Now note that the $\Delta^d$-action is trivial on $P^d_A \subset P^d_{A,\infty}$ for $A = A_\inf/\mu$ or $A = \calO_C$. Thus, the first summand above is identified $R\Gamma_{conts}(\Delta^d,\Z_p) \widehat{\otimes}_{\Z_p} P^d_A$ by base change. As $H^i_{conts}(\Delta^d,\Z_p)$ is a finite free $\Z_p$-module, we can replace the completed tensor product with an ordinary tensor product; this gives (2). For (3), choose $a_i \in \{a_1,...,a_d\}$ with minimal $p$-adic valuation, so $a_i \in \Z[\frac{1}{p}] - \Z$; after relabelling, we may assume $i=1$. Define a $2$-term complex $M$ as follows:
\[ M = \Big(A \xrightarrow{[\underline{\epsilon}^{a_1}] - 1} A\Big).\]
Then $M$ admits the structure of an $A_\inf/([\underline{\epsilon}^{a_1}] - 1)$-complex. By the choice of $a_1$, it follows that $[\underline{\epsilon}^{a_i}] - 1$ is $0$ in $A_\inf/([\underline{\epsilon}^{a_1}] - 1)$ for all $i$, and thus acts trivially on $M$. This means that $K(A; [\underline{\epsilon}^{a_1}] - 1,...., [\underline{\epsilon}^{a_d}] - 1)$ is a finite direct sum of shifts of $M$, so we are reduced to showing that the homology groups of $M$ are isomorphic to $A/([\underline{\epsilon}^{a_1}] - 1)$. This is clear for $A = \calO_C$ as $[\underline{\epsilon}^{a_1}] - 1$ is a nonzero on $\calO_C$ (since $a_1$ is nonintegral). For $A = A_\inf/\mu$, this follows from the fact that $A_\inf$ is a domain, and $\mu = ([\underline{\epsilon}^{a_1}] - 1) \cdot g$ for some $g \in A_\inf$. Finally, (4) follows instantly from (3) as $[\underline{\epsilon}^a] - 1$ divides $[\underline{\epsilon}^{\frac{1}{p}}] - 1$ for $a \in \Z[\frac{1}{p}] - \Z$.
\end{proof}

As promised earlier, our strategy is to study smooth formal $\calO_C$-schemes in terms of \'etale toric co-ordinates. Thus, the following class of smooth formal schemes plays a key role:

\begin{definition}
Fix a formally smooth $\calO_C$-algebra $R$. A {\em framing} of $R$ is a finite \'etale map $\square:P^d \to R$. If such a framing exists, then $\Spf(R)$ (or $R$) is called {\em small}. 
\end{definition}

\begin{remark}
In \cite[Definition 8.5]{BMSMainPaper}, the framings are only required to be \'etale, and not necessarily finite. The stronger hypothesis above is equally ubiquitous in practice (see Lemma~\ref{lem:FramingsExist}), and makes certain arguments involving completed tensor products flow more easily.
\end{remark}

We next check that there exist enough small affine opens on any smooth formal scheme:

\begin{lemma}[Kedlaya]
\label{lem:FramingsExist}
If $\frakX/\calO_C$ is a smooth formal scheme, then there exists a basis of small affine opens $\calU \subset \frakX$. 
\end{lemma}
\begin{proof}
Fix some non-zero non-unit $a \in \calO_C$. By the liftability of \'etale maps across pro-infinitesimal thickenings, it is enough to show that the smooth $\calO_C/a$-scheme $\calX_a := \frakX \otimes_{\calO_C} \calO_C/a$  admits a basis of affine opens that are finite \'etale over $\mathbb{T}^d_{\calO_C/a}$. By standard finite presentation arguments, it suffices to show the same for the smooth $k$-scheme $\frakX_a \otimes_{\calO_C/a} k \simeq \frakX \otimes_{\calO_C} k$. Now \cite[Theorem 2]{KedlayaAffine} (setting $m=0$) shows that an analogous statement is true if one replaces the torus $\mathbb{T}^d_k$ by the affine space $\A^d_k$, so it is enough to show that each $x \in \A^d(k)$ has a Zariski open neighbourhood $U_x$ such that $U_x$ is isomorphic to $\mathbb{T}^d_k$. Such a $U_x$ can be chosen, for example, by translating $x$ to $(1,1,...,1) \in \A^d(k)$ and then removing the co-ordinate hyperplanes.
\end{proof}

Given a framed algebra, one can lift the constructions in Notation~\ref{not:Torus} along the framing as follows:

\begin{construction}[Deforming a framed algebra to $A_\inf$]
Given a framed algebra $(R,\square)$, we construct the following auxiliary rings and group actions: 
\begin{enumerate}
\item The perfectoid algebra $R_\infty^\square := P^d_\infty \otimes_{P^d} R$; here the tensor product is automatically completed as $P^d \to R$ is finite \'etale. The $P^d$-linear $\Delta^d$-action on $P^d_\infty$ induces an $R$-linear $\Delta^d$-action on $R^\square_\infty$.
\item For any $A \in \{\calO_C, A_\inf, A_\inf/\mu\}$, the framing $\square$ deforms uniquely\footnote{In this construction, we implicitly use the following fact repeatedly: if $S$ is a ring that $I$-adically complete for a finitely generated ideal $I$, then the category of finite \'etale $S$-algebras and the category of finite \'etale $S/I$-algebras are equivalent (see \cite[\S 1]{GabberAffineAnalog}). In particular, finite \'etale covers of a formal scheme are the same as those of the underlying reduced scheme, and automorphism groups on both sides match up.} to give a finite \'etale $P^d_A$-algebra $R^\square_A$.  
\item For any $A \in \{\calO_C, A_\inf, A_\inf/\mu\}$, there is a ``perfectoid'' version $R^\square_{A,\infty} := P^d_{A,\infty} {\otimes}_{P^d_A} R_A^\square$ relative to $A$. For $A = A_\inf$, we have the formula $R^\square_{A_\inf,\infty} \simeq A_\inf(R_\infty^\square)$.
\item The formula $R^\square_{A_\inf,\infty} \simeq A_\inf(R_\infty^\square)$ implies that the $\Delta^d$-action on $R_\infty^\square$ lifts uniquely to an $A_\inf$-linear $\Delta^d$-action on $R_{A_\inf,\infty}^\square$; here we may use Remark~\ref{rmk:DefThyAinf} to see the uniqueness and existence of the lift.

\item The formula $R^\square_{A_{\inf}/\mu,\infty} = P^d_{A_\inf/\mu,\infty} \otimes_{P^d_{A_\inf/\mu}} R^\square_{A_{\inf/\mu}}$ implies that we have natural $R^\square_{A_\inf/\mu}$-linear $\Delta^d$-action on $R^\square_{A_\inf/\mu,\infty}$ lifting the one from (1).

\item The $\Delta^d$-action in (4) reduces modulo $\mu$ to the $\Delta^d$-action in (5); this follows from Remark~\ref{rmk:DefThyAinf}.

\end{enumerate}
\end{construction}

\begin{remark}[The almost purity theorem]
\label{rmk:APTCohomology}
Let $(R,\square)$ be a framed algebra, and let $R \to S$ be a finite extension that is \'etale after inverting $p$. Write $S^\square_\infty := R^\square_\infty \otimes_R S$ for the base change to $R_\infty$; note that this base change is already $p$-adically complete. Then the classical formulation of Faltings' almost purity theorem (see \cite[Theorem 3.1]{FaltingsJAMS}, \cite[\S 2b]{FaltingsAlmostEtale}) states that the map $R^\square_\infty \to S^\square_\infty$ is almost finite \'etale\footnote{This result is best viewed as a $p$-adic analog of Abhyankar's lemma. To explain this, recall that one geometric consequence of the latter is the following: if $R_0$ is smooth $\mathbf{C}$-algebra, and $f \in R_0$ is a nonzerodivisor with $R_0/f$ also smooth, then any finite \'etale cover of $R_0[\frac{1}{f}]$ extends uniquely to a finite \'etale cover of $R_0$, at least after adjoining an $n$-th root of $f$. In particular, if we set $R = \colim_n R_0[f^{\frac{1}{n}}]$, then $R_{f\et} \simeq R[\frac{1}{f}]_{f\et}$. The almost purity theorem is the mixed characteristic analog of this result when $f = p$ and $R$ is a perfectoid $\mathcal{O}_C$-algebra: one cannot extend finite \'etale covers to finite \'etale covers in this setting, but can do so in the almost sense.}. In particular, if $R[\frac{1}{p}] \to S[\frac{1}{p}]$ was Galois with group $G$, then the canonical map
\[ R^\square_\infty \to R\Gamma_{conts}(G, S^\square_\infty)\] 
is an almost isomorphism in the sense of Definition~\ref{def:AlmostIsomOC}. This particular consequence of the almost purity theorem is the essential one for the calculations that follow.
\end{remark}

\subsection{Almost purity and \'etale cohomology}
\label{ss:APT}

We now relate the smooth formal $\calO_C$-schemes introduced \S \ref{ss:FramedAlgebras} to their generic fibres. All results presented here are contained \cite{ScholzePAdicHodge}, and many of them are originally due to to Faltings \cite{FaltingsAlmostEtale} in different language and under some stronger hypotheses on integral models; we stick to the former for compatibility with \cite{BMSMainPaper} and wider applicability. 

Let $\frakX/\calO_C$ be a smooth formal scheme, and write $X$ for the generic fibre (viewed as an adic space in the sense of Huber \cite{HuberContVal,HuberDefAdic}); we write $X_{an}$ for the category of open subsets of $X$, viewed as a site in the usual way. To this data, Scholze has attached in \cite[\S 3]{ScholzePAdicHodge} the pro-\'etale site $X_{\proet}$, whose objects are (roughly) towers of finite \'etale covers of open subsets of $X_{an}$. This site comes equipped with a morphism $\mu:X_{\proet} \to X_{an}$ defined by observing an open subset of $X$ is an object of $X_{\proet}$. Huber's sheaf $\mathcal{O}_{X_{an}}^+$ on $X_{an}$ from \cite{HuberDefAdic} then defines a sheaf $\calO_X^+ := \mu^* \mathcal{O}_{X_{an}}^+$ of $\Z_p$-algebras on $X_{\proet}$, and we write $\widehat{\calO_X^+}$ for its $p$-adic completion. One of the key results of the theory is the following \cite[\S 3]{FaltingsAlmostEtale}, \cite[\S 4]{ScholzePAdicHodge}, stated somewhat imprecisely:

\begin{theorem}[Locally perfectoid nature of $X_{\proet}$]
\label{thm:ProetaleLocallyPerfectoid}
For any affinoid $U \in X_\proet$, there exists a cover $V \to U$ in $X_{\proet}$ with $V$ affinoid such that $\widehat{\calO_X^+}(V)$ is a perfectoid $\calO_C$-algebra, and $H^i(V_{\proet}, \widehat{\calO_X^+}) \stackrel{a}{\simeq} 0$ for $i > 0$.
\end{theorem}

The basic example of this theorem is captured in the following example:

\begin{example}
Say $\frakX = \Spf(R)$ is small, and $\square:P^d \to R$ is a framing. As $R \to R^\square_\infty$ is the $p$-adic completion of a direct limit of $R$-algebras which are \'etale after inverting $p$, the map $\Spf(R^\square_\infty) \to \Spf(R)$ induces a pro-\'etale cover $U \to X$ on generic fibres, the value of $\widehat{\calO_X^+}(U) \simeq R^\square_\infty$ is perfectoid.
\end{example}

Since $\widehat{\calO_X^+}$ is locally perfectoid, one can apply Fontaine's functor $A_\inf(-)$ (as in Definition~\ref{def:Ainf}) to the values of $\widehat{\calO_X^+}$ to obtain a sheaf $A_{\inf,X}$ of $A_\inf$-algebras on $X_{\proet}$ (see \cite[\S 6]{ScholzePAdicHodge}) The almost purity theorem (or, rather, the consequence recalled in Remark~\ref{rmk:APTCohomology}) tells us what the cohomology of these sheaves on certain objects looks like (in the sense of almost mathematics):

\begin{theorem}[The pro-\'etale cohomology of framed algebras]
\label{thm:APTCohomologySmall}
Assume $\frakX = \Spf(R)$ is small, and $\square:P^d \to R$ is a framing. Then there is a canonical almost isomorphism
\[ R\Gamma_{conts}(\Delta^d, R^\square_{\infty})  \stackrel{a}{\simeq}  R\Gamma(X_\proet, \widehat{\calO_X^+}) \]
and
\[  R\Gamma_{conts}(\Delta^d, R^\square_{A_\inf,\infty})  \stackrel{a}{\simeq} R\Gamma(X_\proet, A_{\inf,X}).\]
\end{theorem}

By Lemma~\ref{lem:GroupCohKoszulCohPerfectoidTorus}, both complexes above are quite ``large''; for instance, the group $H^1(R\Gamma_{conts}(\Delta^d, P^d_\infty))$ has at least countably many nonzero distinct direct summands. Nevertheless, in the global case, there is enough cancellation in the relevant local-to-global spectral sequence (see \cite[Tag 01ES]{StacksProject}) to give a finite dimensional answer \cite[Corollary 1 to Theorem 8]{FaltingsAlmostEtale}, \cite[Theorem 5.1]{ScholzePAdicHodge}:

\begin{theorem}[The primitive comparison theorem]
\label{thm:PrimitiveCompThm}
Assume $\frakX$ is proper. Then there is a natural almost isomorphism
\[ R\Gamma(X_{\proet}, \Z_p) \otimes^L_{\Z_p} \calO_C \stackrel{a}{\simeq} R\Gamma(X_{\proet}, \widehat{\calO_X^+})\]
and
\[ R\Gamma(X_{\proet}, \Z_p) \otimes^L_{\Z_p} A_\inf \stackrel{a}{\simeq} R\Gamma(X_{\proet}, A_{\inf,X}).\]
\end{theorem}

We do not know if Theorem~\ref{thm:PrimitiveCompThm} is true in the real (i.e., non-almost) world.

\begin{remark}
Consider the nearby cycles map $\nu:(X_{\proet}, \widehat{\calO_X^+}) \to (\frakX_\Zar, \calO_\frakX)$: this is simply the map\footnote{The nomenclature is explained as follows: if $\frakX$ is the $p$-adic completion of a smooth $\calO_C$-scheme $\calX$, then the complex $R\nu_* \Z_\ell$ calculates the nearby cycles of the constant sheaf $\Z_\ell$ for the map $\calX \to \Spec(\calO_C)$ by \cite[Theorem 0.7.7]{Huber}.} of ringed topoi obtained by taking generic fibres of the formal schemes. Using this map, one can reformulate Theorem~\ref{thm:APTCohomologySmall} (resp. Theorem~\ref{thm:PrimitiveCompThm}) as describing the values (i.e., hypercohomology), in the almost sense, of the complexes $R\nu_* \widehat{\calO_X^+}$ and $R\nu_* A_{\inf,X}$ when $\frakX$ is small (resp. $\frakX$ is proper). This perspective will be relevant for the sequel as we will modify the complex $R\nu_* A_{\inf,X}$ in a suitable way to prove our main theorem.
\end{remark}

Theorem~\ref{thm:PrimitiveCompThm} is quite surprising at first glance: the sheaf $\widehat{\calO_X^+}$ is the completion of $\calO_X^+$, and the latter is a pro-\'etale sheafified version of the structure sheaf $\calO_X^+$ on $X$. Thus, one might naively expect that $H^*(X_\proet, \widehat{\calO_X^+})[\frac{1}{p}]$ resembles the coherent cohomology $H^*(X,\calO_X)$, so it must have cohomological dimension $\leq$ $\dim(X)$. On the other hand, Theorem~\ref{thm:PrimitiveCompThm} tells us that we instead obtain $p$-adic \'etale cohomology, which has cohomological dimension $2 \dim(X)$. This apparent confusion is resolved by noticing that $H^*(X_\et,\calO_X^+)$ has {\em a lot} of torsion; this torsion builds up to copies of $\calO_C$ under $p$-adic completion, thereby accounting for the missing cohomological degrees.


\begin{remark}[Algebro-geometric reconstruction of $R\Gamma(X_\proet, \calO_X^+)$]
\label{rmk:ProetaleBeilinson}
Assume that $C := \C_p := \widehat{\overline{\Q_p}}$, and that $\frak{X}$ is the $p$-adic completion of a proper smooth $\overline{\Z_p}$-scheme $\calX$. In this case, the complex $K = R\Gamma(X_\et, \calO_X^+)$ has the following features: $K[\frac{1}{p}]$ calculates $H^*(\mathcal{X},\mathcal{O}_{\mathcal{X}}) \otimes_{\overline{\Z_p}} C$, while $\widehat{K}[\frac{1}{p}]$ calculates $H^*(\mathcal{X}_C, \mathbf{Z}_p) \otimes_{\Z_p} C$. The algebro-geometrically minded reader might wonder if it is possible to construct such a complex $K$ directly, i.e., without recourse to any nonarchimedean geometry or perfectoid spaces. It turns out that it is indeed possible to do so. In fact, one can show (but this is not relevant to the sequel) that $K \simeq M \otimes_{\overline{\Z_p}} \calO_C$, where $M$ is the value on $(\calX[\frac{1}{p}],\calX)$ of the $h$-sheafification of the presheaf $(U,\overline{U}) \mapsto R\Gamma(\overline{U},\calO_{\overline{U}})$ for the $h$-topology on Beilinson's category of arithmetic pairs over $\overline{\Q}_p$ (see \cite[\S 2.2]{Beilinsonpadic} for the category of pairs, \cite[\S 3.3]{Beilinsonpadic} for an identification of the $p$-adic completion of $M$ with the $\calO_C$-\'etale cohomology of $X$, and \cite[\S 3.4]{Beilinsonpadic} for an identification of $M[\frac{1}{p}]$ with the coherent cohomology of the structure sheaf $\calX[\frac{1}{p}]$.)
%
\end{remark}

\begin{remark}[Explicitly constructing $R\Gamma(X_{\proet}, \mathcal{O}_X^+)$ for abelian varieties with good reduction]
Continuing the theme and notation of Remark~\ref{rmk:ProetaleBeilinson}, we will explain a direct construction of the complex $M$ appearing above in the special case where $\calX/\overline{\Z_p}$ is an abelian scheme. Let $[p^n]:\calX_n \to \calX$ be the multiplication by $p^n$ map on the abelian scheme $\calX$, so the finite group $T_n := \calX(C)[p^n]$ acts by translation on $\calX_n$ compatibly with $[p^n]$ for the trivial action on the target. As $n$ varies, we get a tower
\[ ... \to \calX_{n+1} \xrightarrow{[p]} \calX_n \xrightarrow{[p]} .... \xrightarrow{[p]} \calX_0 = \calX.\]
Passing to cohomology, this defines a complex
\[ M := \colim_n R\Gamma_{conts}(T_n, R\Gamma(\calX_n, \calO_{\calX_n})) \simeq R\Gamma_{conts}(T, R\Gamma(\calX_\infty, \calO_{\calX_\infty})),\]
where $\calX_\infty := \lim \calX_n$ (as schemes), and $T = \lim T_n$ is the $p$-adic Tate module of $\calX$. As $[p^n]$ is finite \'etale in characteristic $0$, there is a canonical map $M \to R\Gamma(X_\et, \calO_X^+)$, and we will check that the induced map $\alpha:M \otimes_{\overline{\Z_p}} \calO_C \to R\Gamma(X_\et, \calO_X^+)$ is an isomorphism. Let $X_n = \calX_n[\frac{1}{p}]$ for simplicity, so $X$ is the analytification of $X_0 \otimes_{\overline{\Q_p}} \C_p$. As $[p^n]$ is \'etale after inverting $p$, we have
\[ R\Gamma(X_0, \calO_{X_0}) \simeq R\Gamma_{conts}(T_n, R\Gamma(\calX_n, \calO_{X_n}))\]
for all $n$ by Galois descent. Thus, all the transition maps in the colimit defining $M$ are isomorphisms after inverting $p$, so $\alpha[\frac{1}{p}]$ is an isomorphism. Also, note that $[p]^*$ induces $p^i$ on $H^i(\calX,\calO_{\calX}) \simeq \wedge^i H^1(\calX,\calO_\calX)$. As $p$-adic completion kills $\Z[\frac{1}{p}]$-modules, the $p$-adic completion of $R\Gamma(\calX_\infty,\calO_{\calX_\infty}) \simeq \colim R\Gamma(\calX_n, \calO_{\calX_n})$ is identified with $\calO_C[0]$, and thus
\[ \widehat{M} \simeq R\Gamma_{conts}(T, \calO_C).\]
Now abelian varieties are $K(\pi,1)$ for the pro-\'etale topology with fundamental groups given by the Tate modules, so $\widehat{\alpha}$ is also an isomorphism, and hence $\alpha$ is an isomorphism.
\end{remark}

\section{The decalage functor}
\label{sec:Leta}

The main homological ingredient that goes into Theorem~\ref{thm:MainThm} as well as \cite{BMSMainPaper} is the ``d\'ecalage'' functor of Ogus \cite[\S 8]{BerthelotOgus}. This functor turns out to give a systematic way for killing some torsion in the derived category, and will be leveraged in the sequel to get rid of the ``largeness'' alluded to following Theorem~\ref{thm:APTCohomologySmall}. For the rest of this section, fix a ring $A$ and a regular element $f \in A$. The key (underived) functor is:

\begin{definition}
For any chain complex $K^\bullet$ of $f$-torsionfree $A$-modules, define a subcomplex $\eta_f(K^\bullet) \subset K^\bullet[\frac{1}{f}]$ by setting
\[ \eta_f(K^\bullet)^i := \{\alpha \in f^i K^i \mid d(\alpha) \in f^{i+1} K^{i+1}\}.\]
\end{definition}

The association $K^\bullet \to \eta_f(K^\bullet)$ is an endo-functor of the category of $f$-torsionfree $A$-complexes. This functor behaves understandably on cohomology:

\begin{lemma}
\label{lem:HomologyLeta1}
For a chain complex $K^\bullet$ of $f$-torsionfree $A$-modules, there is a functorial identification $H^i(\eta_f(K^\bullet)) \simeq H^i(K^\bullet)/H^i(K^\bullet)[f]$; equivalently, this is also identified with the image $f H^i(K^\bullet)$ of multiplication by $f$ on $H^i(K^\bullet)$.
\end{lemma}
\begin{proof}
If $\alpha \in K^i$ is a cycle, then $f^i \alpha \in \eta_f(K^\bullet)^i$ is also a cycle. One easily checks that this association gives the desired isomorphism
\[ H^i(K^\bullet)/H^i(K^\bullet)[f] \simeq H^i(\eta_f K^\bullet).\]
The second part is formal as multiplication by $f$ identifies $H^i(K^\bullet)/H^i(K^\bullet)[f]$ with $f H^i(K^\bullet)$.
\end{proof}


\begin{remark}[Independence of choice of generator $f \in (f)$]
\label{rmk:IgnoreTwistsLeta}
The definition of $\eta_f$ only depends on the ideal $(f) \subset A$, and not its chosen generator. Moreover, the identification in Lemma \ref{lem:HomologyLeta1} is functorial in $K^\bullet$; however, this identification depends on choice of $f$. More canonically, we can write
\[ H^i(\eta_f(K^\bullet)) \simeq (H^i(K^\bullet)/H^i(K^\bullet)[f]) \otimes_A (f^i),\]
where $(f^i) \subset A[\frac{1}{f}]$ denotes the fractional $A$-ideal generated by $f^i$. When formulated this way, the identification is completely canonical, i.e., dependent only on the ideal $(f)$, and not its chosen generator $f$. In fact, all constructions of this section can be defined in the generality of locally principal ideal sheaves on a ringed topos that might not even posses a global generator (see \cite[\S 6]{BMSMainPaper}); we do not do so here in the interest of simplicity (both visual and calculational) and notational ease. 
\end{remark}

Using Lemma~\ref{lem:HomologyLeta1}, one checks that for any $M \in D(A)$ with a representative $K^\bullet$ with $f$-torsionfree terms, the object $\eta_f(K^\bullet) \in D(A)$ depends only on $M$, and is independent of the choice $K^\bullet$ of representative. Thus, we get:

\begin{proposition}[Definition of $L\eta_f$]
\label{lem:HomologyLeta}
Applying the functor $\eta_f(-)$ to a representative with $f$-torsionfree terms, we obtain a functor $L\eta_f:D(A) \to D(A)$. This functor satisfies $H^i(L\eta_f(K)) \simeq H^i(K)/H^i(K)[f]$ functorially in $K$, and hence commutes with truncations in the derived category.
\end{proposition}

In particular, we obtain a endofunctor of $D(A)$ that kills torsion in homology. However, this pleasant feature comes with a warning: one must be careful when directly applying standard derived categorical intuition to $L\eta$:

\begin{warning}
The functor $L\eta_f$ is not exact. Indeed, with $A = \Z$ and $f = p$, one readily computes that $L\eta_f(\Z/p) = 0$, and yet $L\eta_f(\Z/p^2) = \Z/p \neq 0$.
\end{warning}

To understand the behaviour of $L\eta_f(-)$ better, fix some $K \in D(A)$. It is clear that $L\eta_f(K)[\frac{1}{f}] = K[\frac{1}{f}]$. Thus, to understand $L\eta_f(K)$, we must understand $L\eta_f(K) \otimes_A A/f$. For this, recall that the following construction:

\begin{construction}[The Bockstein construction]
\label{cons:BocksteinComplex}
For any $K \in D(A)$, we have the following associated Bockstein complex:
\begin{equation}
\label{eqn:BocksteinComplexDef}
(H^*(K/f), \beta_f) := \Big(... \to H^i(K \otimes^L_A f^iA/f^{i+1}A) \xrightarrow{\beta_f^i} H^{i+1}(K \otimes^L_A f^{i+1}A/f^{i+2} A) \to ... \Big), 
\end{equation}
where the differential $\beta_f^i$ is the boundary map associated to the exact triangle obtained by tensoring the canonical triangle
\[ f^{i+1}A/f^{i+2}A \to f^iA/f^{i+2}A \to f^iA/f^{i+1}A\]
with $K$. (The necessary identity $\beta^{i+1}_f \circ \beta^i_f = 0$ can be verified directly, or by observing that $\Ext^2_A(A/f,-) = 0$.) As the terms $f^i A$ only depend on the ideal $(f)$, this construction evidently only depends on the ideal $(f) \subset A$, and not its chosen generator $f \in A$. 
\end{construction}

\begin{remark}[The Bockstein construction, non-canonically]
\label{rmk:IgnoreTwistsBockstein}
As in Remark~\ref{rmk:IgnoreTwistsLeta}, we can ignore the twists in Construction~\ref{cons:BocksteinComplex} to simply write
\[ (H^*(K/f), \beta_f) := \Big(... \to H^i(K \otimes^L_A A/f)  \xrightarrow{\beta_f^i} H^{i+1}(K \otimes^L_A  A/f) \to ... \Big),\]
with the differential induced by the boundary map in the triangle
\[ A/f \xrightarrow{f} A/f^2 \xrightarrow{\can} A/f\]
as before; such an identification depends on the choice of $f$ (and not merely on the ideal $(f)$), and will be used without further comment in the sequel when performing calculations.
\end{remark}

An elementary example of this construction is the following:

\begin{example}
\label{ex:BocksteinKillsTorsion}
Fix some $K \in D(A)$ such that $f \cdot H^i(K) = 0$ for all $i$. Then we claim that $(H^*(K/f),\beta_f)$ is acyclic. To see this, note that the assumption on $K$ and the exact triangle
\[ K \stackrel{f}{\to} K \to K \otimes^L_A A/f\]
give a short exact sequence
\[ 0 \to H^i(K) \stackrel{a_i}{\to} H^i(K \otimes^L_A A/f) \stackrel{b_i}{\to} H^{i+1}(K) \to 0\]
for all $i$. Under this identification, the map $\beta_f$ is identified as $a_{i+1} \circ b_i$. Thus, the cycles and boundaries in degree $i$ in the complex $(H^*(K/f), \beta_f)$ coincide the image of $a_i$, which proves acyclicity. 
\end{example}

Example~\ref{ex:BocksteinKillsTorsion} implies that the Bockstein construction kills $f$-torsion, while Proposition~\ref{lem:HomologyLeta} implies that $L\eta_f(-)$ kills $f$-torsion as well. These two phenomenon turn out to be closely related:

\begin{lemma}[$L\eta_f$ lifts the Bockstein]
\label{lem:LetaBockstein}
For $K \in D(A)$, there is a canonical isomorphism
\[ L\eta_f(K) \otimes^L_A A/f \simeq (H^*(K/f),\beta_f)\]
\end{lemma}
\begin{proof}
Choose a representative $K^\bullet$ of $K$ with $f$-torsionfree terms. Given $f^i \alpha \in \eta_f(K^\bullet)^i \subset f^i K^i$, the image of $\alpha$ in $K^i/f$ is a cycle in $K^\bullet/f$, thus yielding a map $\eta_f(K^\bullet)^i \to H^i(K^\bullet/f)$. One can then show that construction defines a map $\eta_f(K^\bullet) \to (H^*(K/f),\beta_f)$ of complexes inducing an isomorphism $\eta_f(K^\bullet)/f \simeq (H^*(K/f),\beta_f)$ of complexes, see \cite[Proposition 6.12]{BMSMainPaper}. 
\end{proof}

\begin{remark}
\label{rmk:LetaAbelian}
The functor $K \mapsto (H^*(K/f), \beta_f)$ on $D(A)$ takes values in the {\em abelian} category of chain complexes of $A$-modules. Using this observation and Lemma~\ref{lem:LetaBockstein}, one can show that the functor $K \mapsto L\eta_f(K) \otimes_A A/f$ on $D(A)$ lifts naturally to  the abelian category of chain complexes; this observation will be useful in the sequel. In contrast, the functor $L\eta_f(-)$ does not naturally lift to chain complexes.
\end{remark}

\begin{remark}[Preservation of commutative algebras]
The functor $L\eta_f(-):D(A) \to D(A)$ is lax symmetric monoidal: this amounts to the assertion that if $M^\bullet$ and $N^\bullet$ are two $K$-flat complexes $A$-modules, then there is a natural map $\eta_f(M^\bullet) \otimes \eta_f(N^\bullet) \to \eta_f(M^\bullet \otimes N^\bullet)$; see \cite[Proposition 6.7]{BMSMainPaper}. In particular, $L\eta_f(-)$ carries commutative algebras to commutative algebras, so $L\eta_f(K) \otimes^L_A  A/f$ acquires the structure of a commutative algebra in $D(A/f)$. Moreover, if $K$ is a commutative algebra in $D(A)$, then $(H^*(K/f),\beta_f)$ acquires the structure of a commutative differential graded algebra of $A/f$-modules via cup products. One can show that the isomorphism of Lemma~\ref{lem:LetaBockstein} intertwines these commutative algebra structures. In fact, one can show that $L\eta$ naturally lifts to a lax-symmetric monoidal functor at the level of derived $\infty$-categories, and thus preserves commutative algebras in the $\infty$-categorical sense (see \cite[\S 2]{LurieHA}), i.e., it carries $E_\infty$-algebras to $E_\infty$-algebras.
\end{remark}

We now prove a series of lemmas describing the behaviour of $L\eta_f(-)$; all of these, except Lemma~\ref{lem:LetaExact} and Lemma~\ref{lem:LetaRegularSequence}, are directly extracted from \cite[\S 6]{BMSMainPaper}. First, observe that the formation of $L\eta_f(-)$ commutes with restriction of scalars:

\begin{lemma}
\label{lem:LetaRestScalars}
Let $\alpha:A \to B$ be a map of rings such that $\alpha(f)$ is a regular element. For any $M \in D(B)$, there is a natural identification
\[ \alpha_* \big(L\eta_{\alpha(f)}(M)\big) \simeq L\eta_f(\alpha_* M).\]
\end{lemma}
\begin{proof}
This immediate from the definition of $L\eta_f$ and the fact that a $B$-module is $\alpha(f)$-torsionfree if and only if it is $f$-torsionfree when regarded as an $A$-module.
\end{proof}

Likewise, it composes well with other operations of a similar nature:

\begin{lemma}
\label{lem:LetaComposition}
Let $g \in A$ be a regular element. Then there is a natural identification $L\eta_f(L\eta_g(M)) \simeq L\eta_{fg}(M)$ for any $M \in D(A)$.
\end{lemma}
\begin{proof}
This follows immediately by writing down the definition of either side.
\end{proof}

The next lemma gives a criterion for the $L\eta_f(-)$ to preserve exact triangles, and is crucial to the sequel:

\begin{lemma}
\label{lem:LetaExact}
Fix an exact triangle 
\[ K \to L \to M\]
in $D(A)$. Assume that the boundary map $H^i(M/f) \to H^{i+1}(K/f)$ is the $0$ map for all $i$. Then the induced sequence 
\[ L\eta_f(K) \to L\eta_f(L) \to L\eta_f(M)\] is also an exact triangle.
\end{lemma}
\begin{proof}
Since the exactness is clearly true after inverting $f$, it suffices to check exactness after applying $(-)/f$. Using Lemma~\ref{lem:LetaBockstein}, we must check that the sequence of chain complexes
\[ (H^*(K/f),\beta_f) \to (H^*(L/f),\beta_f) \to (H^*(M/f),\beta_f) \]
gives an exact triangle in $D(A)$. But an even stronger statement is true: the preceding sequence is exact in the abelian category of chain complexes (and thus also in $D(A)$) by the assumption on the boundary maps.
\end{proof}

We give one example of the preceding lemma being used to pass from the almost world to the real world; this mirrors some computations from the sequel.

\begin{example}[$L\eta$ takes some almost isomorphisms to isomorphisms]
Let $A = \mathcal{O}_C$, and $f = p$. Let $K \in D(\calO_C)$ be perfect, and $K \to L$ be a map in $D(\calO_C)$ whose cone $M$ is almost zero. Then $L\eta_p(K) \simeq L\eta_p(L)$. To see this using Lemma~\ref{lem:LetaExact}, we must show that the boundary maps $H^i(M/p) \to H^{i+1}(K/p)$ are $0$ for all $i$. Now each $H^i(M/p)$ is an almost zero module (as $M$ is almost zero). On the other hand, as $K$ is perfect, so is $K/p$, so each $H^i(K/p)$ is a finitely presented torsion $\mathcal{O}_C$-module. Any such module is isomorphic to a finite direct sum of $\mathcal{O}_C$-modules of the form $\mathcal{O}_C/g$ for suitable $g \in \mathcal{O}_C$. One then checks easily (using the valuation on $\mathcal{O}_C$) that $H^i(K/p)$ has no almost zero elements, so the boundary maps have to be $0$.
\end{example}

Using the previous criterion, we obtain one for $L\eta_f(-)$ to commute with reduction modulo a different element.

\begin{lemma}
\label{lem:LetaRegularSequence}
Fix some nonzerodivisor $g \in A$. For any $K \in D(A)$, there is a natural map
\[ \alpha:L\eta_f(K) \otimes^L_A A/g \to L\eta_f(K \otimes^L_A A/g).\]
This map is an isomorphism if $H^*(K \otimes^L_A A/f)$ has no $g$-torsion.
\end{lemma}

The first version of \cite{BMSAnnouncement} had an erroneous comment to the effect that the conclusion of this lemma holds without any hypothesis on $H^*(K/f)$.

\begin{proof}
The canonical map $L\eta_f(K) \to L\eta_f(K \otimes^L_A A/g)$ induces $\alpha$ by base change (as the target is an $A/g$-complex). For the second part, consider the triangle
\[ K \xrightarrow{g} K \to K \otimes^L_A A/g.\]
The boundary map $H^i(K \otimes^L_A A/g \otimes^L_A A/f) \to H^{i+1}(K \otimes^L_A A/f)$ is $0$ for all $i$ as $H^*(K \otimes^L_A A/f)$ has no $g$-torsion. Lemma~\ref{lem:LetaExact} applied this triangle then proves the desired claim.
\end{proof}

%
%

In the sequel, we will repeatedly apply $L\eta_f$ to certain Koszul complexes built from endomorphisms of modules that enjoy good divisibility properties with respect to $f$. Thus, for later use, we record the following fundamental example:

\begin{example}[$L\eta_f(-)$ simplifies Koszul complexes]
\label{ex:LetaKoszul}
Fix elements $g_1,...,g_d \in A$, and let $M$ be an $f$-torsionfree $A$-module. Let $K : =K(M; g_1,...,g_d)$. The complex $L\eta_f(K)$ can be understood directly in the following two situations:
\begin{enumerate}
\item If $f$ divides each $g_i$, then $L\eta_f(K)$ is identified with $K(M;\frac{g_1}{f},...,\frac{g_d}{f})$. When $d=1$, this follows immediately from the definition of $\eta_f$. The general case is deduced by induction; see \cite[Lemma 7.9]{BMSMainPaper} for details. 
\item If some $g_i$ divides $f$, then $L\eta_f(K) \simeq 0$. Indeed, we may assume (after relabelling) that $g_1 \mid f$. In this case, each $H^i(K)$ is killed by $g_1$, and thus by $f$. Proposition~\ref{lem:HomologyLeta} then implies that $L\eta_f(K) \simeq 0$.
\end{enumerate}
\end{example}

The next lemma provides a handy criterion for mapping into $L\eta_f(M)$ for a suitable $M$:

\begin{lemma}
\label{lem:LetaFactorize}
Let $K \in D^{\leq 1}(A)$ and $M \in D^{\geq 0}(A)$ with $H^0(M)$ being $f$-torsionfree. Fix a map $\alpha:K \to M$.
\begin{enumerate}
\item The canonical map $L\eta_f(M) \to M$ induces an isomorphism $H^1(L\eta_f M) \simeq f H^1(M)$.
\item The map $\alpha$ factors as $K \stackrel{\alpha'}{\to} L\eta_f(M) \stackrel{\can}{\to} M$ if and only the map $H^1(\alpha):H^1(K) \to H^1(M)$ has image contained in $fH^1(M)$. Moreover, such a factorization is unique if it exists.
\end{enumerate}
\end{lemma}
\begin{proof}
Any map $K \to M$ factors uniquely as a map $K \to \tau^{\leq 1} M$. As $L\eta$ commutes with truncations, we may assume $M \in D^{[0,1]}(A)$ with $H^0(M)$ being $f$-torsionfree. In this case, one readily constructs a commutative diagram
\[ \xymatrix{H^1(M \otimes_A A/f)[-2] \simeq H^1(M)/fH^1(M)[-2] \ar@{=}[d] \ar[r] & L\eta_f(M) \ar[r]^-{a} \ar[d] & M \ar[d] \\
H^1(M \otimes_A A/f)[-2] \simeq H^1(M)/fH^1(M)[-2] \ar[r] & H^0(M \otimes^L_A A/f)[0] \ar[r]^-{b} & M \otimes_A A/f }\]
with exact rows, where the map $L\eta_f(M) \to H^0(M \otimes_A A/f)[0]$ comes from Lemma~\ref{lem:LetaBockstein}, and the other maps are the obvious ones. The long exact sequence for the top row then gives (1). Now fix a map $\alpha:K \to M$. Such a map factors through $a$ if and only if the induced map $K \to M \otimes_A A/f$ factors through $b$, and the same holds true for the set of choices of such a factorization. The long exact sequence obtained by applying $\Hom(K,-)$ to the bottom row then shows that $\alpha$ factors through $a$ if and only the induced map $K \stackrel{\alpha}{\to} M \to H^1(M)/fH^1(M)[-1]$ is the $0$ map, and that the set of choices for such a factorization is controlled by the set of maps $K \to H^1(M)/fH^1(M)[-2]$; as $K \in D^{\leq 1}$, the latter is trivial, and the former means exactly that the map $H^1(\alpha)$ has image in $fH^1(M)$, so we are done. 
\end{proof}

The $L\eta_f(-)$ functor behaves well with respect to complete objects:

\begin{lemma}
\label{lem:LetaCompletion}
Fix some finitely generated ideal $I \subset A$, and some $K \in D(A)$ such that $K$ is $I$-adically complete. Then $L\eta_f(K)$ is also $I$-adically complete.
\end{lemma}

Recall that $I$-adic completeness for complexes is always meant in the derived sense (see \S \ref{ss:Conventions} for a summary).

\begin{proof}
We must check that $H^i(L\eta_f(K))$ is $I$-adically complete for all $i$. But $H^i(L\eta_f(K)) = H^i(K)/H^i(K)[f]$. Now $I$-adically complete $A$-modules form an abelian category inside all $A$-modules and are closed under kernels, cokernels, and images; the claim follows immediately since a complex $I$-adically complete if and only if its homology groups are so.
\end{proof}

On bounded complexes, the $L\eta_f(-)$ is a bounded distance away from the identity:

\begin{lemma}
\label{lem:LetaInverse}
Assume $K \in D^{[0,d]}(A)$ with $H^0(K)$ being $f$-torsionfree. Then there are natural maps $L\eta_f(K) \to K$ and $K \to L\eta_f(K)$ whose composition in either direction is $f^d$. 
\end{lemma}
\begin{proof}
We may represent $K$ by a chain complex $M^\bullet$ with $M^i$ being $f$-torsionfree, and $M^i = 0$ for $i \notin [0,d]$. There is then an obvious inclusion $\eta_f(M^\bullet) \subset M^\bullet$. Moreover, it is also clear that multiplication of $f^d$ on either complex factors over this inclusion, proving the claim. 
\end{proof}

\section{The complex $\widetilde{\Omega}_{\frakX}$}
\label{sec:TildeOmega}

The goal of this section is to attach a perfect complex $\widetilde{\Omega}_\frakX \in D(\frakX, \mathcal{O}_{\mathfrak{X}})$ to a smooth formal scheme $\frakX/\calO_C$ such that the cohomology sheaves of $\widetilde{\Omega}_{\frakX}$ are given by differential forms; this complex will be responsible for the Hodge-Tate specialization of the cohomology theory $R\Gamma_A(\frakX)$ of Theorem~\ref{thm:MainThmPrecise}, as alluded to in Remark~\ref{rmk:HodgeTate}.

We begin in \S \ref{ss:FramedCalculationTildeOmega} with some calculations of the complex $\widetilde{\Omega}_\frakX$ for small $\frakX$. These calculations are then leveraged to define and prove basic properties of $\widetilde{\Omega}_\frakX$ in \S \ref{ss:TildeOmegaNoFrame}; notably, the calculation of the cohomology sheaves (in the almost sense) is given in Proposition~\ref{prop:TildeOmegaDescription}.

\subsection{Some calculations with framings} 
\label{ss:FramedCalculationTildeOmega}

We begin with a description of the structure of $\widetilde{\Omega}_\frakX$ for a formal torus:

\begin{lemma}
\label{lem:TorusTildeOmega}
Fix an integer $d \geq 0$, and let $\widetilde{\Omega}_{P^d} := L\eta_{\epsilon_p - 1} R\Gamma_{conts}(\Delta^d, P^d_\infty)$. The canonical map induces an isomorphism
\[ P^d \otimes^L_{\calO_C} L\eta_{\epsilon_p - 1} R\Gamma_{conts}(\Delta^d, \calO_C) \to L\eta_{\epsilon_p - 1} R\Gamma_{conts}(\Delta^d, R^\square_\infty) =: \widetilde{\Omega}^{P^d}.\]
In particular, $H^1(\widetilde{\Omega}_{P^d})$ is free of rank $d$, and cup products induce isomorphisms $\wedge^i H^1(\widetilde{\Omega}_{P^d}) \simeq H^i(\widetilde{\Omega}_{P^d})$ for all $i$. Thus, the graded $P^d$-algebra $H^*(\widetilde{\Omega}_{P^d})$ is an exterior algebra on the free $P^d$-module $H^1(\widetilde{\Omega}_{P^d})$.
\end{lemma}
\begin{proof}
Lemma~\ref{lem:GroupCohKoszulCohPerfectoidTorus} (2) and (4) give
\[ R\Gamma_{conts}(\Delta^d, P^d_\infty) \simeq \Big(R\Gamma_{conts}(\Delta^d,\Z_p) \otimes_{\Z_p} P^d\Big) \oplus E\]
where $E$ is some $\calO_C$-complex with homology killed by $\epsilon_p - 1$. Applying $L\eta_{\epsilon_p - 1}$ and using Proposition~\ref{lem:HomologyLeta},  we get
\begin{equation}
\label{eq:TorusTildeOmega}
 L\eta_{\epsilon_p - 1} R\Gamma_{conts}(\Delta^d, P^d_\infty) \simeq P^d {\otimes}^L_{\calO_C} L\eta_{\epsilon_p - 1} R\Gamma_{conts}(\Delta^d, \calO_C).
 \end{equation}
By the Koszul presentation for the group cohomology of $\Delta^d$, the object $R\Gamma_{conts}(\Delta^d,\calO_C)$ is calculated by the complex $M := \oplus_i \wedge^i (\calO_C^{\oplus d})[-i]$. Applying $L\eta_{\epsilon_p - 1}$ to $M$ produces the subcomplex
\[ L\eta_{\epsilon_p -1}(M) := \big(\oplus_i (\epsilon_p -1)^i \wedge^i(\calO_C^{\oplus d})[-i]\big) \hookrightarrow \big(\oplus_i \wedge^i (\calO_C^{\oplus d})[-i]\big) =: M.\]
The left side is easily seen to be $\oplus_i \wedge^i ((\epsilon_p - 1) \calO_C^{\oplus d})[-i]$ by the definition of $L\eta$; this proves all claims by flat base change along $\calO_C \to P^d$.
\end{proof}

The previous calculation passes up along a framing as well:

\begin{corollary}
\label{cor:FramedAlgebraTildeOmega}
Let $\square:P^d \to R$ be a framing, and let $\widetilde{\Omega}_R := L\eta_{\epsilon_p -1} R\Gamma_{conts}(\Delta^d, R^\square_\infty)$. The canonical map gives an isomorphism
\[ R \otimes^L_{\calO_C} L\eta_{\epsilon_p - 1} R\Gamma_{conts}(\Delta^d, \calO_C) \to L\eta_{\epsilon_p - 1} R\Gamma_{conts}(\Delta^d, R^\square_\infty) =: \widetilde{\Omega}_R\]
In particular, $H^1(\widetilde{\Omega}_R)$ is a free of rank $d$, and cup products induce isomorphisms $\wedge^i H^1(\widetilde{\Omega}_R) \simeq H^i(\widetilde{\Omega}_R)$ for all $i$. Thus, the graded $R$-algebra $H^*(\widetilde{\Omega}_{R})$ is an exterior algebra on the free $R$-module $H^1(\widetilde{\Omega}_{R})$.
\end{corollary}
\begin{proof}
Applying flat base change for $R\Gamma_{conts}(\Delta^d,-)$ and $L\eta_{\epsilon_p - 1}$ along the flat map $\square$, we see that
\[ M \simeq R \otimes^L_{P^d} L\eta_{\epsilon_p - 1} R\Gamma_{conts}(\Delta^d, P^d_\infty).\]
Lemma~\ref{lem:TorusTildeOmega} then finishes the proof.
\end{proof}

\begin{remark}
The proof above identifies  $H^1(\widetilde{\Omega}_R)$  with $(\epsilon_p - 1) H^1_{conts}(\Delta^d, R^\square_\infty)$.
\end{remark}

\subsection{Almost local results without framings}
\label{ss:TildeOmegaNoFrame}

We now give the definition of $\widetilde{\Omega}_\frakX$ in general, and prove the comparison with differential forms; all calculations in this section, especially Proposition~\ref{prop:TildeOmegaDescription} and the following discussion, should be interpreted in the almost sense.

Let $\frakX/\calO_C$ be a smooth formal scheme with generic fibre $X$, and let $\nu:\Shv(X_{\proet}) \to \Shv(\frakX_\Zar)$ be the nearby cycles map from \S \ref{ss:APT}. Then we define:

\begin{definition}
Let $\widetilde{\Omega}_{\frakX} := L\eta_{\epsilon_p - 1}(R \nu_* \widehat{\calO_X^+}) \in D(\frakX, \calO_{\frakX})$.
\end{definition}

\begin{remark}
As we see below, $\widetilde{\Omega}_\frakX$ is a perfect complex on $\frakX$ whose cohomology groups are given by differential forms on $X$ (up to twists). This complex is not always a direct sum of its cohomology sheaves: a concrete obstruction is the failure of $\frakX$ to be liftable along the square-zero thickening $A_\inf/\ker(\theta)^2 \to \calO_C$ (see \cite[Remark 8.4]{BMSMainPaper}). It would be interesting to find a direct construction of $\widetilde{\Omega}_{\frakX}$ that does not pass through the generic fibre and $p$-adic Hodge theory. In particular, it is not clear if a variant of $\widetilde{\Omega}_{\frakX}$ exists when the base field is discretely valued.
\end{remark}

Recall that the $p$-adic Tate module of an abelian group $A$ is defined as $T_p(A) = \lim A[p^n]$, where the transition maps are given by multiplication by $p$; when $M$ is $p^\infty$-torsion and $p$-divisible, one  checks that $T_p(M)[1]$ identifies with the $p$-adic completion of $M$. To identify the complex $\widetilde{\Omega}_{\frakX}$ introduced above, recall the following result of Fontaine \cite[Theorem 1]{FontaineDiffForm}, \cite[\S 1.3]{Beilinsonpadic}:

\begin{theorem}[Fontaine]
The $\calO_C$-module $T_p(\Omega^1_{\calO_C/\Z_p})$ is free of rank $1$. Moreover, the map $d\log:\mu_{p^\infty}(\calO_C) \to \Omega^1_{\calO_C/\Z_p}$ induces, on passage to Tate modules, an injective map
\[ \calO_C(1) := \Z_p(1) \otimes_{\Z_p} \calO_C \to T_p(\Omega^1_{\calO_C/\Z_p})\]
whose cokernel is killed by $(\epsilon_p - 1)$. 
\end{theorem}

Fontaine's calculation permits the introduction of a slight modification of the Tate twist as follows:

\begin{definition}[Breuil-Kisin twists]
\label{def:BKTwist}
For any $\calO_C$-module $M$ and $n \in \Z$, write $M\{n\} := M \otimes_{\calO_C} \calO_C\{n\}$, where $\calO_C\{1\} := T_p(\Omega^1_{\calO_C/\Z_p})$, and $\calO_C\{n\} = \calO_C\{1\}^{\otimes n}$.
\end{definition}

\begin{remark}[Explicitly describing $\mathcal{O}_C\{1\}$]
\label{rmk:GeneratorBKTwist}
Let us view $\underline{\epsilon} \in \Z_p(1)$ as a generator. The map $d\log:\Z_p(1) \to \calO_C\{1\}$ then defines an element $d\log([\underline{\epsilon}]) \in \calO_C\{1\}$. This element is (uniquely) divisible by $\epsilon_p - 1$. In fact, the compatible system
\[ \frac{1}{\epsilon_p - 1} \cdot d\log(\epsilon_{p^n}) \in \Omega^1_{\calO_C/\Z_p}[p^n]\] 
of differential forms provides a generator $\omega \in \calO_C\{1\}$ such that $(\epsilon_p - 1) \cdot \omega = d\log([\underline{\epsilon}]$; we will denote the dual generator of $\calO_C\{-1\}$ by $\omega^\vee$ in the sequel.
\end{remark}

\begin{remark}[The cotangent complex of a perfectoid algebra and Breuil-Kisin twists]
\label{rmk:PerfectoidAbsoluteCC}
Let $\C_p = \widehat{\overline{\Q_p}}$ be the completed algebraic closure of $\Q_p$, so $\C_p \subset C$. Fix a perfectoid $\calO_C$-algebra $R$. This gives us maps
\[ \Z_p \stackrel{a}{\to} \overline{\Z_p} \stackrel{b}{\to} \calO_{\C_p} \stackrel{c}{\to} \calO_C \stackrel{d}{\to} R,\]
where $\overline{\Z_p}$ is the integral closure over $\Z_p$ in $\C_p$ (or, equivalently, in $\overline{\Q_p}$), the map $b$ is the $p$-adic completion map, and the rest are the obvious ones. The maps $b$, $c$, and $d$ are flat and relatively perfect modulo $p$, and thus must have a vanishing $p$-adically completed cotangent complex (see Remark~\ref{rmk:PerfectoidCC}). The map $a$ is an inductive limit of lci and generically \'etale maps, so $L_a \simeq \Omega^1_a$. Moreover, as $C$ is algebraically closed, each element of $\overline{\Z_p}$ has a $p$-th root, so the $\overline{\Z_p}$-module $\Omega^1_a$ is $p$-torsion and $p$-divisible, giving
\[ \widehat{L_a} \simeq \calO_{\C_p}\{1\}[1].\]
Combining this with the cotangent complex vanishing shown earlier, we learn that
\[ \widehat{L_{R/\Z_p}} \simeq R\{1\}[1].\]
This perspective also provides another description of $R\{1\}$ as follows. One has $\widehat{L_{A_\inf(R)/\Z_p}} \simeq 0$ since $A_\inf(R)/p \simeq R^\flat$ is perfect, so the transitivity triangle for $\Z_p \to A_\inf(R) \stackrel{\theta}{\to} R$ collapses to give an isomorphism
\[ \widehat{L_{R/A_\inf(R)}} \simeq \widehat{L_{R/\Z_p}}.\]
The right side is $R\{1\}[1]$, while the left side is $\ker(\theta)/\ker(\theta)^2$ by Lemma~\ref{lem:AinfTheta}. Thus, we get
\[ R\{1\} \simeq \ker(\theta)/\ker(\theta)^2.\] 
When $R = \calO_C$, the resulting map $\ker(\theta) \to T_p(\Omega^1_{\calO_C/\Z_p})$ carries $\mu \in \ker(\theta)$ to $d\log([\underline{\epsilon}])$, so $\xi \in \ker(\theta)$ maps to $\omega$.
\end{remark}

We arrive at our promised result identifying the cohomology sheaves of $\widetilde{\Omega}_{\frakX}$:

\begin{proposition}
\label{prop:TildeOmegaDescription}
There is a canonical almost isomorphism $\Omega^1_{\frakX/\calO_C}\{-1\} \stackrel{a}{\simeq} \calH^1(\widetilde{\Omega}_{\frakX})$. Via cup products, this induces an almost isomorphism
\[ \Omega^i_{\frakX/\calO_C}\{-i\} \stackrel{a}{\simeq} \calH^i(\widetilde{\Omega}_{\frakX}).\]
\end{proposition}
\begin{proof}
We first summarize the strategy informally: by a local calculation, we reduce to showing the claim about $\calH^1$. In this case, using the vanishing of the cotangent complex of perfectoids, we construct a natural map $L_{\frakX/\Z_p}\{-1\}[-1] \to R\nu_* \widehat{\calO_X^+}$. Finally, we check that this map induces the desired isomorphism on $\calH^1$ using the criterion in Lemma~\ref{lem:LetaFactorize}.

In more detail, thanks to Corollary~\ref{cor:FramedAlgebraTildeOmega}, it suffices to construct the canonical isomorphism $\Omega^1_{\frakX/\calO_C}\{-1\} \stackrel{a}{\simeq} \calH^1(\widetilde{\Omega}_{\frakX})$. For this, observe that the map $\nu:(X_{\proet},\calO_X^+) \to (\frakX_\Zar,\calO_\frakX)$ of ringed topoi induces a pullback map
\[ \nu^*:L_{\frakX/\Z_p} \to R\nu_* \widehat{L_{{\calO_X^+}/\Z_p}},\]
where the completion on the right is $p$-adic. As $\widehat{\calO_X^+}$ is locally perfectoid on $X_{\proet}$ (see Theorem~\ref{thm:ProetaleLocallyPerfectoid}), Remark~\ref{rmk:PerfectoidAbsoluteCC} gives
\[ \widehat{L_{\calO_X^+/\Z_p}} \simeq \widehat{\calO_X^+}\{1\}[1].\]
Thus, after twisting, the map $\nu^*$ above induces a map
\[ \alpha: \widehat{L_{\frakX/\Z_p}}\{-1\}[-1] \to R\nu_* \widehat{\calO_X^+}.\]
We will check that this map factors uniquely through $\widetilde{\Omega}_{\frakX} := L\eta_{\epsilon_p - 1}R\nu_* \widehat{\calO_X^+} \to R\nu_* \widehat{\calO_X^+}$, and that the induced map 
\[ \alpha': \widehat{L_{\frakX/\Z_p}}\{-1\}[-1] \to \widetilde{\Omega}_{\frakX}\] 
induces the desired almost isomorphism after $p$-adic completion on $\calH^1$. 

We begin by calculate the source of $\calH^1(\alpha)$. Observe that the canonical map 
\[ \widehat{L_{\mathfrak{X}/\mathbf{Z}_p}}\{-1\}[-1] \to \widehat{L_{\mathfrak{X}/\mathcal{O}_C}}\{-1\}[-1]\]
induces an isomorphism on $\mathcal{H}^1$: the fiber is $\widehat{L_{\mathcal{O}_C/\mathbf{Z}_p}} \otimes_{\mathcal{O}_C} \mathcal{O}_{\mathfrak{X}}\{-1\}[-1]$, which lives in $D^{\leq 0}$ by Remark~\ref{rmk:PerfectoidAbsoluteCC}. Moreover, $L_{\mathfrak{X}/\calO_C} \simeq \Omega^1_{\frakX/\calO_C}$ as $\frakX/\calO_C$ smooth. In particular, this is a free $\calO_\frakX$-module, and hence is already $p$-adically complete. Thus, we learn that $\calH^1(\widehat{L_{\mathfrak{X}/\mathbf{Z}_p}}\{-1\}[-1]) \simeq \Omega^1_{\frakX/\calO_C}\{-1\}$.

We now return to checking the above assertion for the map $\alpha$. As $\alpha$ satisfies the hypotheses of Lemma~\ref{lem:LetaFactorize}, we are reduced to checking $\calH^1(\alpha)$ induces an almost isomorphism between $\calH^1(L_{\frakX/\calO_C}\{-1\}[-1]) = \Omega^1_{\frakX/\calO_C}\{-1\}$ and $(\epsilon_p - 1) R^1 \nu_* \widehat{\calO_X^+}$. This is a local assertion, so we may assume $\frakX = \Spf(R)$ is small. Using base change for continuous group cohomology as in Corollary~\ref{cor:FramedAlgebraTildeOmega}, we may even assume $R = P^d$. Thus, it suffices to check that $\Omega^1_{P^d/\calO_C} \otimes \calO_C\{-1\}$ maps isomorphically to $(\epsilon_p - 1) H^1_{conts}(\Delta^d, P^d_\infty)$ under the map $\alpha$ above (using Theorem~\ref{thm:APTCohomologySmall} to almost calculate $R^i \nu_* \widehat{\calO_X^+}$); note that both are free $P^d$-modules of rank $d$ (by Lemma~\ref{lem:TorusTildeOmega} for the latter). Moreover, both sides are graded by $\Z[\frac{1}{p}]^d$, and the map respects the grading by functoriality. Thus, it suffices to show the claim for each graded component. Using this observation, one reduces to the case $d=1$, i.e., we must check that the generator $d\log(t) \otimes \omega^\vee  \in \Omega^1_{P^1/\calO_C}\{-1\}$ (with notation as in Remark~\ref{rmk:GeneratorBKTwist}) is carried isomorphically by $\alpha$ onto the generator $(\epsilon_p - 1) \cdot t^0 \in (\epsilon_p - 1) H^1_{conts}(\Delta, P^1_\infty)$ coming from Lemma~\ref{lem:TorusTildeOmega}. Untwisting, we must check that $\alpha\{1\}$ carries $d\log(t) \in \Omega^1_{P^1/\calO_C}$ to $(\epsilon_p - 1) \cdot t^0 \otimes \omega \in H^1_{conts}(\Delta, P^1_\infty) \otimes \calO_C \{1\}$. This simplifies to 
\[ (\epsilon_p - 1) \cdot t^0 \otimes \omega = t^0 \otimes (\epsilon_p - 1) \cdot \omega = t^0 \otimes d\log([\underline{\epsilon}]).\]
In particular, this element lies in the subspace
\[ H^1_{conts}(\Delta, \Z_p \cdot t^0) \otimes \Z_p(1) \subset H^1_{conts}(\Delta, P^1_\infty) \otimes \calO_C\{1\}.\]
Thus, we must check that $\alpha\{1\}$ carries $d\log(t)$ to $t^0 \otimes d\log(\underline{\epsilon}) \in H^1_{conts}(\Delta, \Z_p) \otimes \Z_p(1)$. This follows from unwinding the definition of $\alpha$, and using that the generator $g \in \Delta$ (coming from the choice of $\underline{\epsilon}$) satisfies the following transformation law when acting on $L_{P^1_\infty/\Z_p}$:
\[ (g-1)(d\log(t)) = g(d\log(t)) - d\log(t) = d\log([\underline{\epsilon}] \cdot t) - d\log(t) = d\log([\underline{\epsilon}]),\]
where the last equality uses that $d\log$ carries multiplication to addition; see \cite[Proposition 8.17]{BMSMainPaper} for more details.
\end{proof}

\begin{remark}
Set $R = P^1$. Proposition~\ref{prop:TildeOmegaDescription} gives a canonical isomorphism 
\[ \Omega^1_{R/\calO_C}\{-1\} \simeq R \otimes^L_{\calO_C} (\epsilon_p - 1) H^1_{conts}(\Delta, \calO_C).\]
Untwisting, this may be viewed as a canonical isomorphism
\[ \Omega^1_{R/\calO_C} \simeq R \otimes^L_{\calO_C} (\epsilon_p - 1) H^1_{conts}(\Delta, \calO_C\{1\}).\]
Recalling that $(\epsilon_p - 1) \calO_C\{1\}$ is is naturally identified with $\calO_C(1)$, we can rewrite this as
\[ \Omega^1_{R/\calO_C} \simeq R\otimes_{\Z_p} H^1_{conts}(\Delta, \Z_p(1)).\]
Unwinding the proof above shows that this isomorphism carries $d\log(t) \in \Omega^1_{R/\calO_C}$ to the element $1 \otimes \can$, where $\can:\Delta \simeq \Z_p(1)$ is the natural identification, viewed as a $1$-cocycle on $\Delta$ valued in $\Z_p(1)$.
\end{remark}

\begin{remark}[Explicitly describing $\widetilde{\Omega}_{\frakX}$]
\label{rmk:TildeOmegaValues}
The proof of Proposition~\ref{prop:TildeOmegaDescription} gives a slightly finer statement. To explain this, define an object $\widetilde{\Omega}^{pre}_\frakX$ of the derived category of presheaves of $\calO_\frakX$-modules on the site $\frakX_\Zar^{sm}$ of small affine opens $\frakU \subset \frakX$ defined by attaching to such an  $\frakU \subset \frakX$ the complex $L\eta_{\epsilon_p - 1} R\Gamma_{\proet}(U, \widehat{\calO_U^+})$, where $U = \frakU_\eta$ is the generic fibre. Thus, $\widetilde{\Omega}_{\frakX}$ is (by definition) the sheafification of $\widetilde{\Omega}^{pre}_{\frakX}$. In particular, for each $\frakU \in \frakX_\Zar^{sm}$, there is a canonical map
\[ \widetilde{\Omega}^{pre}_\frakX(\frakU) := L\eta_{\epsilon_p - 1}R\Gamma_{\proet}(U, \widehat{\calO_U^+}) \to \widetilde{\Omega}_\frakX(\frakU) := R\Gamma(\frakU, \widetilde{\Omega}_{\frakX}).\]
The proof above shows that this map is an almost isomorphism: the cohomology presheaves of $\widetilde{\Omega}^{pre}_\frakX$ were almost identified as locally free sheaves on $\frakX$ when restricted to small affines. In other words, the sheafification process (after applying $L\eta$) is not necessary if one restricts attention to small opens in $\frakX$ and works in the almost world. For notational convenience, given a small open $\frakU := \Spf(R) \subset \frakX$, we will write $\widetilde{\Omega}_R$ for the common value of either complex displayed above in the almost world. Note that this value is $p$-adically complete by Lemma~\ref{lem:LetaCompletion}. 
\end{remark}

\begin{remark}[Excising almost mathematics]
\label{rmk:TildeOmegaNoAlmost}
The almost isomorphism constructed in Proposition~\ref{prop:TildeOmegaDescription} is actually an isomorphism on the nose when $\frakX = \Spf(R)$ is small, see \cite[\S 8]{BMSMainPaper}. In particular, the complex $\widetilde{\Omega}_R$ described in Remark~\ref{rmk:TildeOmegaValues} is a perfect complex of $R$-modules on the nose. We provide a slightly different explanation from \cite{BMSMainPaper} for why this is so, but this is not relevant for Theorem~\ref{thm:MainThm}. 
%
%
%
%
Fix a framed algebra $(R,\square)$ (with $\frakX$ and $X$ as above). Let $K = R\Gamma_{conts}(\Delta^d, R^\square_\infty)$, let $L = R\Gamma(X_\proet, \widehat{\calO_X^+})$, and let $K \to L$ be the natural map with cofiber $Q$. Then $Q$ is almost zero by Theorem~\ref{thm:APTCohomologySmall}. Moreover, it is easy to see that $H^*(K/(\epsilon_p - 1))$ has no nonzero almost zero elements, i.e., elements killed by $\fram \subset \calO_C$: as $P^d \to R^\square$ is finite \'etale, this reduces to the analogous question for $P^d$ itself, which can be checked explicitly\footnote{In more detail, Lemma~\ref{lem:GroupCohKoszulCohPerfectoidTorus} reduces us to checking that all $\calO_C$-linear maps $k \to F \oplus \widehat{\oplus}_i \calO_C/(t_i)$ are $0$; here $F$ is a topologically free $\calO_C$-module, the completed sum is indexed by some set $I$,  $t_i \in \calO_C$ is some element dividing $\epsilon_p - 1$, and the completion is $p$-adic. As $F$ is $p$-torsionfree, there are clearly no nonzero maps $k \to F$. Moreover, as each $t_i$ divides $\epsilon_p - 1$, the completed direct sum coincides with the ordinary direct sum (as the latter is killed by $\epsilon_p - 1$). Thus, we are reduced to showing that all $\calO_C$-linear maps $k \to \oplus_i \calO_C/(t_i)$ are $0$. The $i$-th component of such a map gives an element $a_i \in \calO_C/(t_i)$ such that $f \cdot a = 0$ for each $f \in \fram$; considerations of valuations then show that $a_i = 0$, as wanted.} using Lemma~\ref{lem:GroupCohKoszulCohPerfectoidTorus} (3). Lemma~\ref{lem:LetaExact} then shows that $L\eta_{\epsilon_p - 1}(K) \simeq L\eta_{\epsilon_p - 1}(L)$ (as $L\eta_{\epsilon_p - 1}(Q) = 0$ since $Q$ is almost zero). Moreover, from the proof of Proposition~\ref{prop:TildeOmegaDescription}, it is immediate that $\Omega^i_{R/\calO_C}\{-i\} \simeq H^i(L\eta_{\epsilon_p - 1}(K))$, which gives the promised description.
\end{remark}

\section{The complex $A\Omega_{\frakX}$}
\label{sec:AOmega}

The goal of this section is to define the complex $A\Omega_\frakX$, and identify its Hodge-Tate and de Rham specializations in terms of $\widetilde{\Omega}_{\frakX}$ and the de Rham complex $\Omega^\bullet_{\frakX/\calO_C}$ respectively. We begin in \S \ref{ss:AOmegaFramed} by recording a calculation that will be extremely useful in the identification of the Hodge-Tate specialization. Armed with these calculations, we begin \S \ref{ss:AOmegaNonFramed} by defining $A\Omega_\frakX$ and quickly identifying the Hodge-Tate specialization in Proposition~\ref{prop:AOmegaTildeXi}; the latter then permits us to also identify the de Rham specialization in Proposition~\ref{prop:AOmegadR}, finishing this section.

\subsection{Some calculations with framings}
\label{ss:AOmegaFramed}

The main calculation we need is:

\begin{lemma}
\label{lem:NoTorsionWitt}
Let $\square: P^d \to R$ be a framing. Then all homology groups of the complex 
\[ K := R\Gamma_{conts}(\Delta^d,R^\square_{A_\inf,\infty})\otimes_{A_\inf} A_\inf/\mu\]
 are $p$-torsionfree. 
\end{lemma}
\begin{proof}
The complex $K$ is identified with $R\Gamma_{conts}(\Delta^d, R^\square_{A_\inf/\mu, \infty})$. We have a co-cartesian square
\[ \xymatrix{ P^d_{A_\inf/\mu} \ar[r] \ar[d] & P^d_{A_\inf/\mu,\infty} \ar[d] \\
			R^\square_{A_\inf/\mu} \ar[r] & R^\square_{A_\inf/\mu,\infty}}\]
of commutative rings with the vertical maps being finite \'etale. Recall that there is $\Delta^d$-action on $P^d_{A_\inf/\mu,\infty}$ compatible with the trivial action on $P^d_{A_\inf/\mu}$ along the top horizontal map. This induces a similar action on the bottom arrow. As the continuous group cohomology of $\Delta^d$ can be computed by a Koszul complex, we obtain
\[ K \simeq R\Gamma_{conts}(\Delta^d, P^d_{A_\inf/\mu,\infty}) \otimes^L_{P^d_{A_\inf/\mu}} R^\square_{A_\inf/\mu}.\]
As the map $P^d_{A_\inf/\mu} \to R^\square_{A_\inf/\mu}$ is finite \'etale, to show the previous complex has $p$-torsionfree homology groups, we may reduce to the case where $R = P^d$.  Using the presentation given in Lemma~\ref{lem:GroupCohKoszulCohPerfectoidTorus}, we get
\[ R\Gamma_{conts}(\Delta^d, P^d_{A,\infty}) \simeq  \Big(R\Gamma_{conts}(\Delta^d,\Z_p) \otimes^L_{\Z_p} P^d_A\Big) \oplus \widehat{\bigoplus_{(a_1,...,a_d) \in \Z[\frac{1}{p}]^d-\Z^d}} K(A; [\underline{\epsilon}^{a_1}] - 1,...., [\underline{\epsilon}^{a_d}] - 1).\]
Now each homology group of the first summand is a finite free $P^d_{A_\inf/\mu}$-module, and thus evidently has no $p$-torsion. The homology groups of the Koszul complexes appearing in the second summand above are finite direct sums of copies of $A_\inf/([\underline{\epsilon}^{a_i}] - 1)$ for $a_i \in \Z[\frac{1}{p}]$ by Lemma~\ref{lem:GroupCohKoszulCohPerfectoidTorus} (3). Each of these is abstractly isomorphic to $A_\inf/\mu$ via a suitable power of Frobenius, and hence has no $p$-torsion. We now get the desired conclusion using Lemma~\ref{lem:ptorsionDirectSum} below.
\end{proof}

The following was used above:

\begin{lemma}
\label{lem:ptorsionDirectSum}
Let $K_i \in D(\Z_p)$ be a collection of complexes indexed by some set $I$. If each $H^*(K_i)$ is $p$-torsionfree, then the same is true for $H^*$ of $\widehat{\oplus_i} K_i$.
\end{lemma}
\begin{proof}
We first assume that each $K_i$ is concentrated in degree $0$. Thus, $K_i = M_i[0]$ for a flat $\Z_p$-module $M_i$. In this case,  we have
\[ K := \widehat{\oplus_i} K_i \simeq R\lim( (\oplus_i M_i) \otimes^L_{\Z_p} \Z/p^n) \simeq R\lim( (\oplus_i M_i)/p^n) \simeq \lim ((\oplus_i M_i)/p^n), \]
so $K$ is discrete (i.e., concentrated in degree $0$). Moreover, similar reasoning shows that $K \otimes^L_{\Z} \Z/p \simeq \oplus_i M_i/p$ is discrete, so $H^0(K)$ has no $p$-torsion, as wanted. In general, this reasoning shows that the functor $D(\Z_p)^I \to D(\Z_p)$ of taking completed direct sums is exact on the subcategory of $D(\Z_p)^{I}$ that is spanned by complexes $(K_i) \in D(\Z_p)^I$ such that $H^*(K_i)$ is $p$-torsionfree for all $i$. Thus, one reduces the general case to the previous case by writing each $\Z_p$-complex as a direct sum of its cohomology groups.
\end{proof}

The previous calculation gives us the following identification:

\begin{corollary}
Let $\square:P^d \to R$ be a framing. Then we have
\[ L\eta_\mu(R\Gamma_{conts}(\Delta^d, R^\square_{A_\inf,\infty})) / \tilde{\xi} \simeq \tilde{\theta}_* \big(L\eta_{\epsilon_p -1}(R\Gamma_{conts}(\Delta^d, R^\square_\infty))\big).\]
\end{corollary}
\begin{proof}
This follows immediately from Lemma~\ref{lem:NoTorsionWitt} and Lemma~\ref{lem:LetaRegularSequence}.
\end{proof}

\subsection{Almost local results without framings}
\label{ss:AOmegaNonFramed}

We have all the tools necessary to define and study $A\Omega_{\frakX}$. Thus, let $\frakX/\calO_C$ be a smooth formal scheme.

\begin{definition} Set
\[ A\Omega_\frakX := L\eta_\mu R\nu_* A_{\inf,X} \in D(\frakX, A_\inf),\]
viewed as a commutative algebra object.
\end{definition}

We begin by identifying the Hodge-Tate specialization of $A\Omega_\frakX$:

\begin{proposition}[The Hodge-Tate specialization]
\label{prop:AOmegaTildeXi}
There is a canonical almost isomorphism
\[ A\Omega_{\mathfrak{X}}/\tilde{\xi} \stackrel{a}{\simeq}  \widetilde{\Omega}_{\mathfrak{X}},\]
where $\widetilde{\Omega}_{\frak{X}} \in D(\frakX, \calO_\frakX)$ is viewed as an $A_\inf$-complex via the Hodge-Tate specialization $A_\inf \stackrel{\tilde{\theta}}{\to} \calO_C \to \calO_\frakX$.
\end{proposition}
\begin{proof}
By Lemma~\ref{lem:LetaRegularSequence}, there is a canonical map
\[ A\Omega_\frakX/\tilde{\xi} = L\eta_\mu(R\nu_* A_{\inf,X})/\tilde{\xi} \to L\eta_\mu(R\nu_* A_{\inf,X}/\tilde{\xi}) \simeq \tilde{\theta}_\ast L\eta_{\tilde{\theta}(\mu)}(R \nu_* \widehat{\calO_X^+}) =: \widetilde{\Omega}_\frakX, \]
where we use Lemma~\ref{lem:LetaRestScalars}, applied to the map $A_\inf \stackrel{\tilde{\theta}}{\to} \calO_C \to \calO_\frakX$ with $f = \mu$, for the last identification. To show this map is an almost isomorphism, using the criterion in Lemma~\ref{lem:LetaRegularSequence}, it suffices to check that the cohomology sheaves of the $A_\inf/\mu$-complex $(R\nu_* A_{\inf,X})/\mu$ are almost $\tilde{\xi}$-torsionfree. This is a local statement, so we may assume $\frakX = \Spf(R)$ is small; fix a framing $\square:P^d \to R$. The global sections of $(R\nu_* A_{\inf,X})/\mu$ are then almost identified with $R\Gamma_{conts}(\Delta^d, R^\square_{A_\inf/\mu,\infty})$ by Theorem~\ref{thm:APTCohomologySmall}, so it suffices to check that the cohomology groups of this complex have no $\tilde{\xi}$-torsion. But  $\tilde{\xi} = p$ in $A_\inf/\mu$, so the result follows from Lemma~\ref{lem:NoTorsionWitt}.
\end{proof}

\begin{remark}
\label{rmk:AOmegaTildeXi}
By Frobenius twisting, one can rewrite the identification in Proposition~\ref{prop:AOmegaTildeXi} as
\[ \big(L\eta_{\phi^{-1}(\mu)}(R\nu_* A_{\inf,X})\big)/\xi \stackrel{a}{\simeq} \widetilde{\Omega}_{\frakX},\]
where $\widetilde{\Omega}_{\frak{X}} \in D(\frakX, \calO_\frakX)$ is viewed as an $A_\inf$-complex via $A_\inf \stackrel{\theta}{\to} \calO_C \to \calO_\frakX$.
\end{remark}

\begin{remark}
\label{rmk:AOmegaValues}
Using an analog of the argument (and notation) given in Remark~\ref{rmk:TildeOmegaValues}, together with the observation that a presheaf of $A_\inf$-complexes on the site of small opens $\mathfrak{U} \subset \frakX$ that takes on $\xi$-adically complete values and is a sheaf modulo $\xi$ (i.e., it comes from an object of the derived category of sheaves on $\frakX_{\Zar}^{sm}$ via the forgetful functor) is already a sheaf, one checks the following: for $\frakX = \Spf(R)$ small equipped with a framing $\square:P^d \to R$, the complex $A\Omega_R := R\Gamma(\frakX, A\Omega_\frakX)$ is almost calculated by $L\eta_\mu R\Gamma_{conts}(\Delta^d, R^\square_{A_\inf,\infty})$. In particular, this value is $(p,\mu)$-adically complete by Lemma~\ref{lem:LetaCompletion}.
\end{remark}

\begin{remark}
\label{rmk:AOmegaAinfInverse}
The discussion in Remark~\ref{rmk:AOmegaValues} shows that $A\Omega_{\frakX}$ has non-zero cohomology sheaves only in degrees $0,...,d = \dim(X)$, and that $\calH^0(A\Omega_{\frakX})$ is torsion free (everything in the almost sense). Lemma~\ref{lem:LetaInverse} then shows that there is a canonical map $A\Omega_\frakX \to R\nu_* A_{\inf,X}$ with an inverse up to $\mu^d$ in the almost sense.
\end{remark}

We can now describe the de Rham comparison:

\begin{proposition}[The de Rham specialization]
\label{prop:AOmegadR}
There is a canonical almost isomorphism
\[ A\Omega_{\mathfrak{X}}/\xi \stackrel{a}{\simeq} \Omega^\bullet_{\mathfrak{X}/\calO_C},\]
where the de Rham complex $\Omega^\bullet_{\mathfrak{X}/\calO_C} \in D(\frakX, \calO_C)$ is viewed as an $A_\inf$-complex via $A_\inf \stackrel{\theta}{\to} \calO_C$.
\end{proposition}
\begin{proof}
By Lemma~\ref{lem:LetaComposition}, we have a natural identification
\[ A\Omega_\frakX := L\eta_\mu (R\nu_* A_{\inf,X}) = L\eta_\xi (L\eta_{\phi^{-1}(\mu)}(R \nu_* A_{\inf,X})).\]
Write $M = L\eta_{\phi^{-1}(\mu)}(R\nu_* A_{\inf,X})$ for notational ease. By Lemma~\ref{lem:LetaBockstein}, this gives an identification
\[ A\Omega_\frakX/\xi \simeq (\calH^*(M/\xi), \beta_\xi).\]
Remark~\ref{rmk:AOmegaTildeXi} and Proposition~\ref{prop:TildeOmegaDescription} give an almost isomorphism
\[ \calH^i(M/\xi) \stackrel{a}{\simeq} \Omega^i_{\frakX/\calO_C}\{-i\},\]
where $\calO_\frakX$-modules are viewed as $A_\inf$-modules via $A_\inf \stackrel{\theta}{\to} \calO_C \to \calO_\frakX$. Recall from Construction~\ref{cons:BocksteinComplex}: the $i$-th term of the Bockstein complex $(\calH^*(M/\xi), \beta_\xi)$ is given by 
\[ \calH^i(M \otimes^L_{A_\inf} \xi^iA_\inf/\xi^{i+1} A_\inf) \simeq \calH^i(M/\xi \otimes^L_{A_{\inf}/\xi} \xi^i A_\inf/\xi^{i+1} A_\inf) \simeq \calH^i(M/\xi)\{i\}.\]
Thus, $A\Omega_\frakX/\xi$ is canonically (i.e., as an object of the abelian category of chain complexes, see Remark~\ref{rmk:LetaAbelian}) almost identified with a commutative differential graded algebra (cdga) of the form 
\[ D := \calO_\frakX \to \Omega^1_{\frakX/\calO_C} \to \Omega^2_{\frakX/\calO_C} \to \cdots,\]
with the differentials coming from $\beta_\xi$. To see that the differential coincides with the de Rham differential, we may work locally. Thus, assume $\frakX = \Spf(R)$ for a framed $\calO_C$-algebra $\square:P^d \to R$. As differential forms are local for the \'etale topology, we may assume $R = P^d$. By multiplicativity (i.e., the K\"{u}nneth formula), we may even assume $d = 1$, i.e., $R = \calO_C \langle t^{\pm 1} \rangle$. In this case, the global sections of $M$ are almost given by 
\[ N : =L\eta_{\phi^{-1}(\mu)} R\Gamma_{conts}(\Delta, R^\square_{A_\inf,\infty}).\]
Recall that
\[ R\Gamma_{conts}(\Delta, R^\square_{A_\inf,\infty}) \simeq \widehat{\bigoplus_{a \in \Z[\frac{1}{p}]}} K(A_\inf; [\underline{\epsilon}^a] - 1).  \]
Note that if $a \in \Z[\frac{1}{p}] - \Z$, then $\phi^{-1}(\mu) := [\underline{\epsilon}^{\frac{1}{p}}] - 1$ is divisible by $[\underline{\epsilon}^a] - 1$, and thus the summand above corresponding to $a$ is killed by application of $L\eta_{\phi^{-1}(\mu)}(-)$, thanks to Example~\ref{ex:LetaKoszul}. For $a \in  \Z$, we have $([\underline{\epsilon}^{\frac{1}{p}}] - 1) \mid [\underline{\epsilon}^a] - 1$. By Example~\ref{ex:LetaKoszul}, we can thus write
\[ N = \widehat{\bigoplus_{a \in \Z}}\  K(A_\inf; \frac{[\underline{\epsilon}^a] - 1}{[\underline{\epsilon}^{\frac{1}{p}}] - 1}).\]
For $a \in \Z$, the element $\frac{[\underline{\epsilon}^a] - 1}{[\underline{\epsilon}^{\frac{1}{p}}] - 1}$ is divisible by $\xi := \frac{[\underline{\epsilon}] - 1}{[\underline{\epsilon}^{\frac{1}{p}}] - 1}$, and thus becomes $0$ modulo $\xi$. This yields
\begin{equation}
\label{eq:AOmegamodxiKoszul}
N/\xi \simeq \widehat{\bigoplus_{a \in \Z}} \ K(A_\inf/\xi; 0).
\end{equation}
Proposition~\ref{prop:AOmegaTildeXi} (and also Remark~\ref{rmk:AOmegaValues}) give a canonical identification $N/\xi \simeq \widetilde{\Omega}_R$. Unwrapping the construction, this identification works as follows:
\begin{enumerate}
\item The element $t^a \in R = H^0(\widetilde{\Omega}_R)$ identifies with ``$1$'' in the (cohomological) degree $0$ term of the Koszul complex $K(A_\inf/\xi, 0)$ above in (grading) degree $a$
\item The element $t^a d\log(t) \otimes \omega^\vee \in \Omega^1_{R/\calO_C} \otimes \calO_C\{-1\} = H^1(\widetilde{\Omega}_R)$ identifies with ``$1$'' in the (cohomological) degree $1$ term of the Koszul complex $K(A_\inf/\xi, 0)$ above in (grading) degree $a$ (with $\omega$ as in Remark~\ref{rmk:GeneratorBKTwist}).
\end{enumerate}
As the de Rham differential satisfies $d(t^a) = a \cdot t^a \cdot d\log(t)$, we are reduced to showing the following: if $K = K(A_\inf; \frac{[\underline{\epsilon}^a] - 1}{[\underline{\epsilon}^{\frac{1}{p}}] - 1})$ for some $a \in \Z$, then the Bockstein operator
\[ \beta_\xi:H^0(K/\xi) \to H^1(K/\xi)\]
is identified with multiplication by $a$. The element $1 \in H^0(K/\xi) = A_\inf/\xi$ is carried under $\beta_\xi$ to the element
\[ \Big(\frac{[\underline{\epsilon}^a] - 1}{[\underline{\epsilon}^{\frac{1}{p}}] - 1}\Big) / \xi \cdot 1 =  \Big(\frac{[\underline{\epsilon}^a] - 1}{[\underline{\epsilon}^{\frac{1}{p}}] - 1}\Big) / \Big(\frac{[\underline{\epsilon}] - 1}{[\underline{\epsilon}^{\frac{1}{p}}] - 1}\Big) \cdot 1 = \frac{[\underline{\epsilon}^a] - 1}{[\underline{\epsilon}] - 1} \cdot 1 \in H^1(K/\xi).\]
Now we have
\[ \frac{[\underline{\epsilon}^a] - 1}{[\underline{\epsilon}] - 1} = \sum_{i=0}^{a-1} [\underline{\epsilon}]^i = a \in A_\inf/\xi,\]
so $\beta_\xi$ is identified with multiplication by $a$, as wanted.
\end{proof}

\begin{remark}[$p$-adic Cartier isomorphism]
Fix a smooth formal scheme $\frakX/\calO_C$, and let $X/T$ denote its base change to $T = \calO_C/p$. Consider the two maps $a:A_\inf \stackrel{\theta}{\to} \calO_C \to \calO_C/p$ and $b:A_\inf \stackrel{\tilde{\theta}}{\to} \calO_C \to \calO_C/p$. Since $\tilde{\theta} = \theta \circ \phi^{-1}$, we have the formula $a = \phi \circ b$, and hence we have
\[ \phi_* a^* A\Omega_\frakX \simeq b^* A\Omega_\frakX.\]
Using Proposition~\ref{prop:AOmegadR}, the left side simplifies to $\phi_* \Omega^\bullet_{X/T}$. Using Proposition~\ref{prop:AOmegaTildeXi}, the right side simplifies to $\tilde{\Omega}_{\frakX} \otimes_{\calO_C} \calO_C/p$ in the almost sense. In particular, using Proposition~\ref{prop:TildeOmegaDescription}, we get an almost isomorphism
\[ \calH^i(\phi_* \Omega^\bullet_{X/T}) \simeq \Omega^i_{X/T} \otimes_T T\{-i\},\]
which is a variant of the Cartier isomorphism. We refer the reader to \cite{MorrowNotesIpHT} for a discussion of a lift of this isomorphism to $W_n(\calO_C)$.
\end{remark}

\begin{remark}[Excising almost mathematics]
\label{rmk:AOmegaNoAlmost}
We continue the theme of Remark~\ref{rmk:TildeOmegaNoAlmost}; this is not relevant to the global applications in \S \ref{sec:globalAOmega}. Let $\frakX = \Spf(R)$ be small, and fix a framing $\square:P^d \to R$. Then one can directly show that the canonical map 
\[ \alpha: L\eta_\mu R\Gamma_{conts}(\Delta^d, R^\square_{A_\inf,\infty}) \to L\eta_\mu R\Gamma(X_\proet, A_{\inf,X})\]
is an actual isomorphism, not just an almost isomorphism. In particular, if we write $A\Omega_R$ for the common value of either complex above, then Proposition~\ref{prop:AOmegaTildeXi} (combined with Remark~\ref{rmk:TildeOmegaNoAlmost}) gives an actual isomorphism $A\Omega_R/\tilde{\xi} \simeq \widetilde{\Omega}_R$, while Proposition~\ref{prop:AOmegadR} gives an actual isomorphism $A\Omega_R/\xi \simeq \Omega^\bullet_{R/\calO_C}$. In particular, this argument yields a short (and simpler) proof of all the main results of \cite[\S 9]{BMSMainPaper} (except the variants involving sheaves of Witt vectors). To prove that $\alpha$ is an isomorphism, using an obvious variant of the argument given in Remark~\ref{rmk:TildeOmegaNoAlmost} (with almost mathematics over $A_\inf$ replacing that over $\calO_C$, and the element $\mu$ replacing $\epsilon_p - 1$), we reduce  to checking that $\Hom_{A_\inf}(W(k), H^*(K/\mu)) = 0$ for $K = R\Gamma_{conts}(\Delta^d, R^\square_{A_\inf,\infty})$. As in the proof of Lemma~\ref{lem:NoTorsionWitt}, this reduces to the case of $P^d$ itself, where it can be checked by hand\footnote{We must show that all $A_\inf$-linear maps $W(k) \to F \oplus \widehat{\oplus}_i A_\inf/([\underline{\epsilon}^{a_i}] - 1)$ are $0$; here $F$ is a topologically free $A_\inf/\mu$-module, the $p$-adically completed direct sum is indexed by a set $I$, and $a_i \in \Z[\frac{1}{p}] - \Z$. Note that the target is derived completion of the corresponding non-completed direct sum by the argument in Lemma~\ref{lem:NoTorsionWitt}. By $p$-adic completeness of the target, it suffices to show the same vanishing when the target is replaced by its mod $p^n$ reduction for all $n$. Moreover, by devissage, we may reduce to the case $n=1$. This reduces us to showing that $\Hom_{\calO_C^\flat}(k, \oplus_i \calO_C^\flat/(t_i)) = 0$, where $t_i \in \calO_C^\flat$ are some elements. This follows by the same argument given in Remark~\ref{rmk:TildeOmegaNoAlmost}.} using Lemma~\ref{lem:GroupCohKoszulCohPerfectoidTorus}.
\end{remark}

\begin{remark}[$q$-de Rham cohomology]
Consider the small algebra $R  = P^1 = \calO_C \langle t^{\pm 1} \rangle$. We have seen above (using Remark~\ref{rmk:AOmegaNoAlmost} to get rid of almost zero ambiguities) that the complex $A\Omega_R$ is described explicitly as follows:
\[ A\Omega_R := \widehat{\bigoplus_{i \in \Z}} \Big( A_\inf \cdot t^i \xrightarrow{\frac{[\underline{\epsilon}^i] - 1}{[\underline{\epsilon}]-1}} A_\inf \cdot t^i \Big).\]
Now recall that 
\[ [n]_q := \frac{q^n - 1}{q-1}\]
is the standard {\em $q$-analog} of the integer $n$. Thus, setting $q = [\underline{\epsilon}]$, we can rewrite the preceding description more suggestively as
\[ A\Omega_R := A_\inf \langle t^{\pm 1} \rangle \xrightarrow{\nabla_q} A_\inf \langle t^{\pm 1} \rangle dt.\]
Here $dt$ is a formal variable such that $t^{-1} dt$ corresponds to $t^0$ in the previous presentation, and $\nabla_q$ is the $q$-analogue of the standard derivative, i.e.,
\[ \nabla_q(t^n) = [n]_q t^{n-1} dt = \frac{q^n - 1}{q-1} t^{n-1} \cdot dt.\]
More canonically, one can write
\[ \nabla_q(f(t)) = \frac{f(qt) - f(t)}{qt-t} \cdot dt.\]
In other words, the complex $A\Omega_R$ can be viewed as a $q$-analog of the de Rham complex: setting $q = 1$ recovers the usual de Rham complex. More generally, a similar formula can be given for $A\Omega_R$ for any framed $\calO_C$-algebra $(R,\square)$, see \cite[Lemma 9.6]{BMSMainPaper}. This description has at least two advantages: (a)  it is easy to write this complex directly in terms of $(R,\square)$ without any knowledge of $p$-adic Hodge theory, and (b) it makes sense in more general situations. Unfortunately, we do not know how to prove directly that this description is independent of the framing $\square$, up to quasi-isomorphism. Moreover, the connection to \'etale cohomology is very mysterious from this perspective. Conjectures about the expected properties of such $q$-de Rham complexes were formulated recently by Scholze \cite{ScholzeqdR}.
\end{remark}

\section{Global results}
\label{sec:globalAOmega}

We now reap the global fruit of the  work done so far. Thus, assume that $C$ is a spherically complete, algebraically closed, nonarchimedean extension of $\Q_p$. Let $\frakX/\calO_C$ be a proper smooth formal scheme. The cohomology theory promised in Theorem~\ref{thm:MainThmPrecise} is constructed from the complex $A\Omega_\frakX$ as follows:

\begin{definition}
\label{def:RGammaA}
\[ R\Gamma_A(\frakX) := R\Gamma(\frakX, A\Omega_\frakX)_\ast \in D(A_\inf).\]
\end{definition}

The main properties of this theory are:

\begin{proposition} 
\label{prop:RGammaAOmega}
One has the following:
\begin{enumerate}
\item Finiteness: $R\Gamma_A(\frakX)$ is a perfect $A_\inf$-complex.
\item de Rham comparison: There is a canonical identification $R\Gamma_A(\frakX) \otimes^L_{A_\inf} A_\inf/\xi \simeq R\Gamma_{dR}(\frakX/\calO_C)$ in $D(\calO_C)$.
\item \'Etale comparison: There is a canonical identification $R\Gamma_A(\frakX)[\frac{1}{\mu}] \simeq R\Gamma(X_{\proet}, \Z_p) \otimes_{\Z_p} A_\inf[\frac{1}{\mu}]$ in $D(A_{\inf}[\frac{1}{\mu}])$.
\item Hodge-Tate comparison: There is a canonical identification $R\Gamma_A(\frakX) \otimes^L_{A_\inf,\tilde{\theta}} \calO_C \simeq R\Gamma(\frakX,\widetilde{\Omega}_\frakX)_*$ in $D(\calO_C)$. In particular, there is an $E_2$ spectral sequence
\[ E_2^{i,j}: H^i(\frakX, \Omega^j_{\frakX/\calO_C}\{-j\}) \Rightarrow H^{i+j}(R\Gamma_A(\frakX) \otimes^L_{A_\inf,\tilde{\theta}} \calO_C)\]
of $\calO_C$-modules.
\end{enumerate}
\end{proposition}
\begin{proof}
Using Remark~\ref{rmk:AOmegaValues}, the complex $R\Gamma(\frakX, A\Omega_\frakX)$ is derived $(p,\mu)$-adically complete in the almost sense. Moreover, by Proposition~\ref{prop:AOmegadR}, the complex $R\Gamma(\frakX, A\Omega_\frakX)/\xi \simeq R\Gamma(\frakX, A\Omega_\frakX/\xi)$ is identified with $R\Gamma_{dR}(\frakX/\calO_C)^a$ in $D(\calO_C)^a$. As the functor $K \mapsto K_*$ preserves complete objects and commutes with reduction modulo $\xi$, we learn that $R\Gamma_A(\frakX)$ is $(p,\mu)$-adically complete, and $R\Gamma_A(\frakX)/\xi \simeq (R\Gamma_{dR}(\frakX/\calO_C)^a)_*$. Since $\frakX$ is proper, the complex $R\Gamma_{dR}(\frakX/\calO_C)$ is perfect, and thus 
\[ R\Gamma_{dR}(\frakX/\calO_C) \simeq (R\Gamma_{dR}(\frakX/\calO_C)^a)_*\]
by Lemma~\ref{lem:SphCompleteVanishingAinf}; this proves (2). Moreover, it shows that $R\Gamma_A(\frakX)$ is $(p,\mu)$-adically complete, and is perfect modulo $\xi$. As $A_\inf$ is $\xi$-adically complete, this formally implies $R\Gamma_A(\frakX)$ is itself perfect, proving (1).  For (3), by Remark~\ref{rmk:AOmegaAinfInverse}, we have a natural map
\[ R\Gamma(\frakX, A\Omega_\frakX) \to R\Gamma(\frakX, R\nu_* A_{\inf,X})^a\]
in $D_{comp}(A_\inf)^a$, and this map has an inverse up to $\mu^d$. Now we also know
\[ R\Gamma(\frakX, R\nu_* A_{\inf,X}) \simeq R\Gamma(X_{\proet}, A_{\inf, X}) \simeq R\Gamma(X_{\proet}, \Z_p) \otimes^L_{\Z_p} A_\inf,\]
in $D_{comp}(A_\inf)^a$, where the last isomorphism is Theorem~\ref{thm:PrimitiveCompThm}. Applying $(-)_*$ to the map considered above, and using the preceding isomorphism, we get a map
\[ R\Gamma_A(\frakX) \to \Big(\big(R\Gamma(X_{\proet}, \Z_p) \otimes_{\Z_p} A_\inf\big)^a\Big)_* \simeq R\Gamma(X_{\proet},\Z_p) \otimes^L_{\Z_p} A_\inf,\]
which has an inverse up to $\mu^d$; here we used Lemma~\ref{lem:SphCompleteVanishingAinf} for the last isomorphism above. Inverting $\mu$ then gives (3). Finally, Proposition~\ref{prop:AOmegaTildeXi} gives the first part of (4) as reduction modulo $\tilde{\xi}$ commutes with application of $(-)_*$ ; for the spectral sequence, we use Proposition~\ref{prop:TildeOmegaDescription} and Remark~\ref{rmk:SphCompletePerfectComplex}.
\end{proof}

\begin{remark}[Avoiding spherical completeness]
\label{rmk:RGammaANoAlmost}
If one is willing to use Remark~\ref{rmk:AOmegaNoAlmost}, then all results in this section are easily seen to be true (not merely almost so) without the assumption that $C$ is spherically complete, provided one uses the complex $R\Gamma_A(\frakX) := R\Gamma(\frakX, A\Omega_\frakX)$; note, if $C$ is not spherically complete, we are not allowed to use Definition~\ref{def:RGammaA} in lieu of the preceding one, since  Lemma~\ref{lem:SphCompleteVanishingAinf} would not be true anymore.
\end{remark}

\begin{remark}[Avoiding the primitive comparison theorem]
\label{rmk:NoPrimitive}
The proof of Proposition~\ref{prop:RGammaAOmega} (3) given above uses the primitive comparison theorem~\ref{thm:PrimitiveCompThm}. In fact, it is possible to give a direct proof of this result using the methods of this paper (more precisely, by using Proposition~\ref{prop:RGammaAOmega} (1)), as we now explain. 

\begin{proof}[Alternate proof of Proposition~\ref{prop:RGammaAOmega} (3)]
First, by the purely algebraic \cite[Lemma 4.26]{BMSMainPaper}, it suffices to produce a comparison isomorphism
\[ R\Gamma(X_\proet, \mathbf{Z}_p) \otimes_{A_{\inf}} W(C^\flat) \simeq R\Gamma_A(\frakX) \otimes_{A_\inf} W(C^\flat)\]
in $D(W(C^\flat))$. For this, if we topologize $W(C^\flat)$ with its $p$-adic topology and $A_\inf$ with the $(p,\xi)$-adic topology, then we have 
\[ R\Gamma_A(\mathfrak{X}) \otimes_{A_{\inf}} W(C^\flat) \simeq R\Gamma_A(\mathfrak{X}) \widehat{\otimes}_{A_{\inf}} W(C^\flat)\] 
via the natural map: the analogous statement is true for any perfect complex over $A_\inf$ (as $A_\inf$ is $p$-adically complete), so the claim follows from Proposition~\ref{prop:RGammaAOmega} (1). As $W(C^\flat)$ is the $p$-adic completion of $A_{\inf}[\frac{1}{\mu}]$, and because $R\Gamma(\frakX,-)$ commutes with filtered colimits and derived limits, it follows that 
\[ R\Gamma(\mathfrak{X}) \otimes_{A_\inf} W(C^\flat) \simeq R\Gamma(\mathfrak{X}, \widehat{A\Omega_\frakX[\frac{1}{\mu}]}).\]
 As inverting $\mu$ kills $L\eta_\mu$, we have $\widehat{A\Omega_\frakX[\frac{1}{\mu}]} = R\nu_* \widehat{A_{\inf,X}[\frac{1}{\mu}]} = R\nu_* W(\calO_X^\flat)$, which gives 
 \[ R\Gamma_A(\frakX) \otimes_{A_\inf} W(C^\flat) \simeq R\Gamma(\frakX, R\nu_* \calO_X^\flat) \simeq R\Gamma(X_{\proet}, W(\calO_X^\flat)),\]
so the complex on the right is a perfect $W(C^\flat)$-complex.  Now, as $\calO_X^\flat$ is a sheaf of perfect rings on $X_{\proet}$, the perfect $W(C^\flat)$-complex $R\Gamma(X_\proet, W(\calO_X^\flat))$ comes equipped with a $\phi$-linear isomorphism $\phi_X:R\Gamma(X_{\proet}, W(\calO_X^\flat)) \simeq R\Gamma(X_{\proet}, W(\calO_X^\flat))$ induced by the Frobenius automorphism $\phi_X$ of $\calO_X^\flat$. Lemma~\ref{lem:UnitRoot} then gives a canonical isomorphism 
\[ R\Gamma(X_\proet, W(\calO_X^\flat))^{\phi_X=1} \otimes_{\mathbf{Z}_p} W(C^\flat) \simeq R\Gamma(X_\proet, W(\calO_X^\flat)).\]
It is thus enough to identify $R\Gamma(X_\proet, W(\calO_X^\flat))^{\phi_X=1}$ with $R\Gamma(X_\proet, \mathbf{Z}_p)$. But this is immediate from the Artin-Schreier sequence $0 \to \mathbf{Z}_p \to W(\calO_X^\flat) \xrightarrow{\phi_X - 1} W(\calO_X^\flat) \to 0$ on $X_\proet$.
\end{proof}

The following is a variant of a classical fact in $p$-linear algebra, and was used above:

\begin{lemma}
\label{lem:UnitRoot}
Let $K$ be an algebraically closed perfect field of characteristic $p$. Let $M$ be a perfect complex of $W(K)$-modules equipped with an isomorphism $\phi_M:M \simeq M$ that is linear over the Frobenius automorphism $\phi:W(K) \simeq W(K)$. Then the homotopy fibre\footnote{The homotopy fibre $F$ of a map $f:M \to N$ in a triangulated category is defined to be $\mathrm{cone}(f)[-1]$ (for some choice of a cone), so there is an exact triangle $F \to M \xrightarrow{f} N$.} $L := M^{\phi_M=1}$ of $\phi_M - 1:M \to M$ is a perfect complex of $\mathbf{Z}_p$-modules, and the natural map $L \otimes_{\mathbf{Z}_p} W(K) \to M$ is an equivalence.
\end{lemma}
\begin{proof}
Note that both $M$ and $L$ are derived $p$-adically complete. Moreover, the formation of $L$ commutes with reduction modulo $p$. As perfectness of a $p$-adically complete $\mathbf{Z}_p$-complex can be checked after reduction modulo $p$, we may assume $M$ is killed by $p$, i.e., $M$ is a perfect complex of $K$-modules equipped with a $\phi$-linear isomorphism $\phi_M:M \simeq M$. In this case, it is well-known (see \cite[Expose III, Lemma 3.3]{CLCrystalline}) that for each $i$, the map $H^i(\phi_M - 1)$ is surjective on $H^i(M)$, and that the kernel $L_i := \ker(H^i(M) \xrightarrow{\phi_M - 1} H^i(M))$ is a finite dimensional $\mathbf{F}_p$-vector space such that $L_i \otimes_{\mathbf{F}_p} K \simeq H^i(M)$ via the natural map. The claim now follows immediately as $L_i = H^i(L)$.
\end{proof}
\end{remark}

\newpage

\bibliographystyle{alpha}
\bibliography{mybib}

\end{document}